\documentclass[11pt, oneside]{amsart}   	
\usepackage{geometry}                		
\usepackage[parfill]{parskip}    		
\usepackage{graphicx}
\usepackage{color}
\usepackage{amssymb,amsmath}
\newtheorem{theorem}{Theorem}
\newtheorem{lemma}[theorem]{Lemma}
\newtheorem{corollary}[theorem]{Corollary}

\theoremstyle{definition}
\newtheorem{remark}[theorem]{Remark}

\newtheorem{definition}[theorem]{Definition}
\usepackage{epstopdf}
\numberwithin{equation}{section}
\numberwithin{theorem}{section}

\author{Bent Fuglede}
\address{Department of Mathematical Sciences, University of Copenhagen, Universitetsparken 5,
2100 Copenhagen, Denmark}
\email{fuglede@math.ku.dk}

\author{Natalia Zorii}
\address{Institute of Mathematics of
National Academy of Sciences of Ukraine, Tereshchenkivska 3, 01601,
Kyiv-4, Ukraine}
\email{natalia.zorii@gmail.com}

\thanks{This paper under the title "Green kernels associated with Riesz kernels" will appear in Annales Academi\ae\ Scientiarum Fennic\ae, Mathematica 43, 2018}

\begin{document}

\title[Riesz kernels and the associated Green kernels]{Balayage for Riesz kernels with application to potential theory for the associated Green kernels}

\begin{abstract}
We study properties of the $\alpha$-Green kernel $g_D^\alpha$ of order $0<\alpha\leqslant2$ for a domain $D\subset\mathbb R^n$, $n\geqslant3$. This kernel is associated with the Riesz kernel $|x-y|^{\alpha-n}$, $x,y\in\mathbb R^n$, in a manner particularly well known in the case $\alpha=2$. Besides the usual principles of potential theory, we establish for the $\alpha$-Green kernel the property of consistency. This allows us to prove the completeness of the cone of positive measures $\mu$ on $D$ with finite energy $g_D^\alpha(\mu,\mu):=\iint g_D^\alpha(x,y)\,d\mu(x)\,d\mu(y)$ in the topology defined by the energy norm $\|\mu\|_{g_D^\alpha}=\sqrt{g_D^\alpha(\mu,\mu)}$, as well as the existence of the $\alpha$-Green equilibrium measure for a relatively closed set in $D$ of finite $\alpha$-Green capacity. The main tool is a generalization of Cartan's theory of balayage (sweeping) for the Newtonian kernel to the $\alpha$-Riesz kernels with $0<\alpha<2$.
\end{abstract}
\maketitle

\section{Introduction}\label{sec1}

The $\alpha$-Riesz kernel $\kappa_\alpha$ of order $0<\alpha<n$, given by $\kappa_\alpha(x,y):=|x-y|^{\alpha-n}$, $x,y\in\mathbb R^n$, was studied first by M.~Riesz \cite{R}, see also Landkof \cite{L}.
Throughout this paper we assume that $0<\alpha\leqslant2$ and $n\geqslant3$, $n\in\mathbb N$.

The main purpose of the present paper is to study properties of the associated $\alpha${\it -Green\/} kernel $g_D^\alpha$ on a domain $D\subset\mathbb R^n$. The kernel $g_D^\alpha(x,y)$ is obtained from the $\alpha$-Riesz kernel $|x-y|^{\alpha-n}$ by subtracting the compensating term, which for given $y\in D$ is $\alpha$-harmonic for $x\in D$ and essentially agrees with the $\alpha$-Riesz kernel off $D$. We show that $g_D^\alpha$ has the basic properties of the classical Green kernel on $D$ (where $\alpha=2$ and $D$ typically is regular in the sense of the solvability of the classical Dirichlet problem). Besides the complete maximum principle in a form which includes Frostman's maximum principle and the domination principle, we establish the energy principle and the property of consistency which were known before in the classical case only. Consistency is a property related to the completeness of the cone of positive measures $\mu$ on $D$ with finite energy $g_D^\alpha(\mu,\mu):=\iint g_D^\alpha(x,y)\,d\mu(x)\,d\mu(y)$ in the topology defined by the energy norm $\|\mu\|_{g_D^\alpha}=\sqrt{g_D^\alpha(\mu,\mu)}$, cf.\ \cite{Fu1}, and it allows us for example to prove the existence of the $\alpha$-Green equilibrium measure for a relatively closed set in $D$ of finite $\alpha$-Green capacity. The results obtained generalize those for the Riesz kernel (see e.g. \cite{L}), corresponding to the case where the $\alpha$-Riesz capacity of the complement $D^c:=\mathbb R^n\setminus D$ equals $0$. Although the theory of $\alpha$-Green potentials recently has obtained an interesting development based on probabilistic arguments, see e.g. \cite{Chen,Ku}, the above-mentioned results are new, and they are obtained in the framework of the classical potential-analytic approach.

Our main tool is the development of H.~Cartan's \cite{Ca2} and Landkof's \cite{L} ideas concerning $\alpha$-Riesz balayage of Radon measures onto closed sets in $\mathbb R^n$. We chiefly draw on Cartan's work, though formulated for $\alpha=2$, because the corresponding results in \cite{L} have not all been completely justified (see Section~\ref{adequate1} below for details). For our purpose, where energy of measures plays a key role, such a generalization is only in part available in the setting of balayage spaces \cite{BH} or $H$-cones \cite{BBC}. In particular, the book on balayage spaces by Bliedtner and Hansen \cite{BH} studies thoroughly restrictions to open subsets, corresponding here to the case of $\alpha$-Green kernels, and contains a section on the $\alpha$-Riesz kernels. However, the notion of energy, decisive for our main results, has been excluded in \cite{BH}, cf.\ the Introduction therein. We have therefore chosen to adopt throughout a classical approach to balayage relative to a function kernel.

In the next Section~\ref{sec-basic} we recall some well-known notions and results from the general theory of potentials of real-valued (signed) Radon measures on a locally compact Hausdorff space $X$ relative to a positive, symmetric, lower semicontinuous (l.s.c.) {\it{kernel\/}} $\kappa: X\times X\to[0,+\infty]$. We denote by $\kappa\mu(x):=\int\kappa(x,y)\,d\mu(y)$ the {\it potential\/} of a signed Radon measure $\mu$ relative to the kernel $\kappa$ (whenever defined).

Such a kernel is said to be {\it positive definite\/} if, for every signed Radon measure $\mu$, the {\it energy\/} $\kappa(\mu,\mu):=\iint\kappa(x,y)\,d\mu(x)\,d\mu(y)$ is ${}\geqslant0$ whenever defined. Then the set $\mathcal E_\kappa$ of all signed measures of finite energy forms a pre-Hilbert space with the energy norm $\|\mu\|_\kappa:=\sqrt{\kappa(\mu,\mu)}$ and the associated inner product, called the {\it mutual energy\/}. In addition to the energy norm topology (also called the {\it strong\/} topology) on $\mathcal E_\kappa$, we have the {\it vague\/} topology, even on all of $\mathfrak M=\mathfrak M(X)$, the linear space of all real-valued signed Radon measures on~$X$. A net $(\mu_s)$ on $\mathfrak M$ converges vaguely to $\mu\in\mathfrak M$ if and only if $\int f\,d\mu_s\to\int f\,d\mu$ for every $f\in C_0(X)$, $C_0(X)$ being the space of all continuous functions on $X$ with compact support.\footnote{When speaking of a continuous numerical function we understand that the values are {\it finite\/} real numbers.}

A positive definite kernel $\kappa$ is said to be {\it consistent\/} if, for every positive measure $\mu$ of finite energy $\|\mu\|_{\kappa}^2=\kappa(\mu,\mu)$, the mutual energy $\kappa(\mu,\nu):=\iint\kappa(x,y)\,d\mu(x)\,d\nu(y)$ is vaguely continuous as a function of the positive measure $\nu$ of energy norm $\|\nu\|_\kappa\leqslant1$. Equivalently, every strong Cauchy sequence on the cone $\mathcal E_\kappa^+$ of positive measures of finite energy converges in that topology to any of its vague cluster points, \cite{Fu1,Fu2}.

In Section~\ref{sec2} we consider the $\alpha$-Riesz kernels $\kappa_\alpha$, $0<\alpha\leqslant2$, on $\mathbb R^n$, with frequent reference to~\cite{L}. We develop the theory of $\alpha$-Riesz balayage of a positive measure $\mu\in\mathfrak M(\mathbb R^n)$ onto a closed set $A\subset\mathbb R^n$, being based mainly on the ideas of Cartan for $\alpha=2$ \cite{Ca2}. The results obtained often seem to coincide with those from~\cite{L}, but in fact they are different, being based on different definitions and hence being obtained by different methods. For example, there is the integral representation
\[\mu^A=\int\varepsilon_y^A\,d\mu(y)\] of the balay\'ee $\mu^A$ of $\mu$ onto $A$ in terms of the balay\'ees $\varepsilon_y^A$ of the unit Dirac measures $\varepsilon_y$. In the literature
this integral representation seems to have been more or less taken for granted, though it has been pointed out in \cite[p.~18, Remarque]{Bou} that it requires that the family $(\varepsilon_y^A)_{y\in\mathbb R^n}$ is $\mu${\it -adequate\/} in the sense of \cite[Section~3, D\'efinition~1]{Bou}; see also counterexamples (without $\mu$-adequacy) in Exercises~1 and~2 at the end of that section. We therefore bring in Section~\ref{adequate1} a proof of this adequacy for $\mu$ carried by $\mathbb R^n\setminus A$.

Having proved the integral representation, we are now in a position to use the relation between $\varepsilon_y^A$ and the $\alpha$-Riesz equilibrium measure $\gamma_{A^*}$ of $A^*$, the inverse of $A$ relative to the unit sphere. The $\alpha$-Riesz equilibrium measure $\gamma_{A^*}$ is treated in the extended sense where it might have infinite $\alpha$-Riesz energy, cf.\ \cite[Chapter~V, Section~1, n$^\circ$\,1]{L}), while the relation between $\varepsilon_y^A$ and $\gamma_{A^*}$ is given with the aid of the Kelvin transform, cf.\ the proof of Theorem~\ref{equiv}. This approach enables us to establish the equivalence of $\alpha$-thinness of $A$ at infinity (alternatively, the equivalence of the existence of the $\alpha$-Riesz equilibrium measure $\gamma_A$ on $A$, treated in the extended sense) with the existence of a non-zero bounded positive measure carried by $\mathbb R^n\setminus A$ for which the total mass decreases strictly under sweeping on $A$ (Theorem~\ref{bal-mass-th}). This ties up with the principle of positivity of mass (Theorem~\ref{cor-mass'}), and will be used in subsequent work of the authors.

The final Section~\ref{sec4} on the $\alpha$-Green kernels, $0<\alpha\leqslant2$, is the main part of the paper. Just as the usual Green function on a domain $D$ in $\mathbb R^n$ is the difference between the fundamental harmonic function and its balay\'ee on the complement $D^c$, the $\alpha$-Green kernel $g_D^\alpha$ on $D$ was introduced and studied by Riesz \cite[Chapter~IV]{R} as follows (see also \cite[p.~263]{L}):
\[g^\alpha_D(x,y)=\kappa_\alpha\varepsilon_y(x)-\kappa_\alpha\varepsilon_y^{D^c}(x)\quad
\text{for all \ }x,y\in D.\]
As easily shown, $g^\alpha_D$ is l.s.c., $>0$, and infinite on the diagonal $x=y$. It is essentially known that $g_D^\alpha$ is symmetric: $g_D^\alpha(x,y)=g_D^\alpha(y,x)$.  Seemingly new properties are (as partly indicated above): $g_D^\alpha$ satisfies the
complete maximum principle (in a form which includes the Frostman maximum principle and the domination principle), the Evans--Vasilesco continuity principle, and the principle of positivity of mass. Finally, $g_D^\alpha$ is `{\it perfect\/}', that is, it satisfies the energy principle (Theorem~\ref{th-pos-def}) and it is consistent (Theorem~\ref{th-cons}), which allows us to show that every relatively closed subset $F$ of $D$ of finite $g_D^\alpha$-capacity has a unique $\alpha$-Green equilibrium measure $\gamma_F$ carried by $F$. We also provide a description of the properties of the $\alpha$-Green equilibrium potential $g^\alpha_D\gamma_F$ and single out its characteristic properties (Theorem~\ref{th-equi}).

\section{Basic notions of potential theory on locally compact spaces}\label{sec-basic}
\subsection{Measures, energies, potentials, capacities}\label{sec-ker}

Given a locally compact (Hausdorff) space $X$, we denote by $\mathfrak M=\mathfrak M(X)$ the linear space of all real-val\-ued (signed) Radon measures $\mu$ on $X$, equipped with the {\it vague\/} (weak$^*$) topology, i.e.\ the topology of pointwise convergence on the space $C_0(X)$ of all continuous functions on $X$ with compact support. The vague topology on $\mathfrak M$ is Hausdorff; hence, a vague limit of any sequence (net) in $\mathfrak M$ is unique (whenever it exists).
These and other notions and results from the theory of measures and integration on a locally compact space, to be used throughout the paper, can be found in \cite{B2,Bou,E2} (see also \cite{Fu1} for a short survey).

We denote by $\mu^+$ and $\mu^-$ the {\it positive\/} and {\it negative} parts, respectively, in the Hahn--Jor\-dan decomposition of a measure $\mu\in\mathfrak M$, by $|\mu|:=\mu^++\mu^-$ its total variation, and by $S^\mu_{X}=S(\mu)$ its (closed) {\it support\/}. A measure $\mu$ is said to be {\it bounded\/} if $|\mu|(X)<+\infty$. Let $\mathfrak M^+=\mathfrak M^+(X)$ stand for the (convex, vaguely closed) cone of all positive $\mu\in\mathfrak M$.

The following well-known fact (see, e.g., \cite[Section~1.1]{Fu1}) will often be used.

\begin{lemma}\label{lemma-semi}Let\/ $\psi:\ X\to(-\infty,+\infty]$ be a lower semicontinuous\/ {\rm(}l.s.c.{\rm)} function which is\/~${}\geqslant0$ unless\/ $X$ is compact. Then\/ $\mu\mapsto\int\psi\,d\mu$ is l.s.c.\ on\/ $\mathfrak M^+$ in the\/ {\rm(}induced\/{\rm)} vague topology.\end{lemma}

By a (positive function) {\it kernel\/} $\kappa$ on $X$ we mean a symmetric l.s.c.\ function $\kappa: X\times X\to[0,+\infty]$ which is strictly positive on the diagonal: $\kappa(x,x)>0$ for all $x\in X$. For (signed Radon) measures $\mu,\nu$ on $X$ the {\it potential\/} $\kappa\mu$ and the {\it mutual energy\/} $\kappa(\mu,\nu)$ are defined by\footnote{When introducing notation about numerical quantities we assume
the corresponding object on the right to be well-defined~--- as a finite real number or~$\pm\infty$.}
\[\kappa\mu(x):=\int\kappa(x,y)\,d\mu(y),\quad\kappa(\mu,\nu):=\iint\kappa(x,y)\,d\mu(x)\,d\nu(y),\]
respectively. For $\mu=\nu$,
the mutual energy $\kappa(\mu,\nu)$ becomes the {\it energy\/} $\kappa(\mu,\mu)$ of $\mu$.
Observe that $\kappa\mu(x)$, $\mu\in\mathfrak M$, is well defined provided $\kappa\mu^+(x)$ and $\kappa\mu^-(x)$ are not both infinite, and then $\kappa\mu(x)=\kappa\mu^+(x)-\kappa\mu^-(x)$. In particular, if $\mu\geqslant0$, then $\kappa\mu$ is defined everywhere and represents a positive l.s.c.\ function on~$X$, cf.\ Lemma~\ref{lemma-semi} above.
Also note that $\kappa(\mu,\nu)$, $\mu,\nu\in\mathfrak M$, is well defined provided
$\kappa(\mu^+,\nu^+)+\kappa(\mu^-,\nu^-)$ or $\kappa(\mu^+,\nu^-)+\kappa(\mu^-,\nu^+)$ is finite.
Let $\mathcal E_\kappa=\mathcal E_\kappa(X)$ consist
of all $\mu\in\mathfrak M$ with $-\infty<\kappa(\mu,\mu)<+\infty$, the latter by definition means that $\kappa(\mu^+,\mu^+)$, $\kappa(\mu^-,\mu^-)$ and $\kappa(\mu^+,\mu^-)$ are all finite (see \cite[Section~2.1]{Fu1}).

For a set $Q\subset X$, let $\mathfrak M^+(Q)$ consist of all $\mu\in\mathfrak M^+$ {\it concentrated on\/} (or {\it carried by\/}) $Q$, which means that $X\setminus Q$ is locally $\mu$-negligible, or equivalently that $Q$ is $\mu$-measurable and $\mu=\mu_Q$ where $\mu_Q$ denotes the trace (restriction) of $\mu$ on $Q$. If $Q$ is closed then $\mu\in\mathfrak M^+$ is concentrated on $Q$ if and only if $S(\mu)\subset Q$. Also note that if either $X$ is {\it countable at infinity\/} (i.e.\ $X$ can be represented as a countable union of compact sets), or $\mu$ is bounded,  then the concept of local $\mu$-negligibility coincides with that of $\mu$-negligibility; and hence $\mu\in\mathfrak M^+(Q)$ if and only if $\mu^*(X\setminus Q)=0$, $\mu^*(\cdot)$ being the {\it outer measure\/} of a set. We denote by $\mathfrak M^+(Q,q)$, $q\in(0,+\infty)$, the (convex) subcone of $\mathfrak M^+(Q)$ consisting of all $\mu$ with $\mu(Q)=q$. Also write $\mathcal E_\kappa^+(Q,q):=\mathcal E_\kappa\cap\mathfrak M^+(Q,q)$, $\mathcal E_\kappa^+(Q):=\mathcal E_\kappa\cap\mathfrak M^+(Q)$, and $\mathcal E_\kappa^+:=\mathcal E_\kappa^+(X)$.

In contrast to \cite{Fu4,Fu5} where a capacity has been treated as a functional acting on positive numerical functions on $X$, in the present study we consider the (standard) concept of capacity as a set function. Thus the ({\it inner\/}) {\it capacity\/} of a set $Q$ relative to the kernel $\kappa$, denoted $c_\kappa(Q)$, is defined as\footnote{Here and in what follows the
infimum over the empty set is taken to be~$+\infty$. We put
$1\bigl/(+\infty)=0$ and $1\bigl/0=+\infty$.}
\begin{equation}\label{cap}1\bigl/c_\kappa(Q):=\inf_{\mu\in\mathcal E_\kappa^+(Q,1)}\,\kappa(\mu,\mu).\end{equation}
Obviously, $0\leqslant c_\kappa(Q)\leqslant+\infty$. Furthermore, by \cite[p.~153, Eq.~2]{Fu1},
\begin{equation}\label{compact}c_\kappa(Q)=\sup\,c_\kappa(K)\quad(K\subset Q, \ K\text{\ compact}).\end{equation}

Throughout the paper, we shall often use the fact that $c_\kappa(Q)=0$ if and only if $\mu_*(Q)=0$ for every $\mu\in\mathcal E_\kappa^+$, $\mu_*(\cdot)$ being the {\it inner measure\/} of a set; cf.\ \cite[Lemma~2.3.1]{Fu1}.

As in \cite[p.\ 134]{L}, we call a measure $\mu\in\mathfrak M^+$ $c_\kappa${\it -absolutely continuous\/} if $\mu(K)=0$ for every compact set $K\subset X$ with $c_\kappa(K)=0$. It follows from (\ref{compact}) that, for such $\mu$, $\mu_*(Q)=0$ whenever $c_\kappa(Q)=0$. Hence every $\mu\in\mathcal E_\kappa^+$ is $c_\kappa$-absolutely continuous, but not conversely, cf.\ \cite[pp.~134--135]{L}.

\subsection{Potential-theoretic principles. Consistency}\label{sec-pr} Among the variety of poten\-tial-theore\-tic principles investigated for example in the comprehensive work by Ohtsuka~\cite{O} (see also the references therein), in the present study we focus mainly on the following four:
 \begin{itemize}\item[\rm(i)] A kernel~$\kappa$ is said to satisfy the {\it continuity principle\/} (Evans--Vasilesco), or to be {\it regular\/} (Choquet \cite{Ch1}) if, for any $\mu\in\mathfrak M^+$ with compact~$S^\mu_{X}$, the potential $\kappa\mu$ is continuous throughout~$X$ whenever its restriction to $S^\mu_{X}$ is continuous.
 \item[\rm(ii)] A kernel $\kappa$ is said to satisfy {\it Frostman's maximum principle\/} if, for any $\mu\in\mathfrak M^+$ with compact support,
 \[\sup_{x\in X}\,\kappa\mu(x)=\sup_{x\in S^\mu_{X}}\,\kappa\mu(x).\]
\item[\rm(iii)] A kernel $\kappa$ is said to satisfy the {\it complete maximum principle} (introduced by Cartan and Deny \cite{CD}) if, for any $\mu\in\mathcal E^+_\kappa$ and $\nu\in\mathfrak M^+$ such that $\kappa\mu\leqslant\kappa\nu+c$ $\mu$-a.e., where $c\geqslant0$ is a constant, the same inequality holds everywhere on $X$.
\item[\rm(iv)] A kernel $\kappa$ is called {\it positive definite\/} if $\kappa(\mu,\mu)\geqslant0$ for every (signed) measure $\mu\in\mathcal E_\kappa$. And $\kappa$ is said to be {\it strictly positive definite\/}, or to satisfy the {\it energy principle\/} if in addition $\kappa(\mu,\mu)>0$ except if $\mu=0$.
\end{itemize}

The complete maximum principle with $c=0$ is called the {\it domination principle\/} (introduced by Cartan \cite{C0} under the name {\it second maximum principle\/}).

The above-mentioned principles are not completely independent of one another. In particular, every kernel satisfying the Frostman maximum principle or the domination principle is positive definite, \cite{N,Ch3}. And for a kernel which is finite off the diagonal and continuous in the extended sense on $X\times X$ we have $\text{\rm(ii)}\Rightarrow\text{\rm(i)}$, see \cite{O1}, \cite{O2}, \cite[Eq.~1.3]{O}, and independently \cite{Ch1}.

In the rest of this section, $\kappa$ is assumed to be positive definite. Then $\mathcal E_\kappa$ forms a pre-Hil\-bert space with the energy seminorm $\|\mu\|_\kappa:=\sqrt{\kappa(\mu,\mu)}$ and the inner product $\kappa(\mu,\nu)$ (see~\cite{Fu1}). The topology
on $\mathcal E_\kappa$ defined by the energy seminorm $\|\cdot\|_\kappa$ is called the {\it strong topology\/}. Clearly, $\|\cdot\|_\kappa$ is a norm if and only if $\kappa$ is strictly positive definite.

Write $\mathcal E^\circ_\kappa:=\bigl\{\mu\in\mathcal E_\kappa^+:\ \|\mu\|_\kappa\leqslant1\bigr\}$. Since $\kappa$ is strictly positive on the diagonal, $\mathcal E_\kappa^\circ$ is {\it vaguely compact\/}, cf.\ \cite[Lemma~2.5.1]{Fu1}.

\begin{definition}\label{def-cons}A (positive definite) kernel $\kappa$ is said to be {\it consistent\/} if, for every $\lambda\in\mathcal E_\kappa^+$, the function $\mu\mapsto\kappa(\lambda,\mu)$, $\mu\in\mathcal E_\kappa^\circ$, is vaguely continuous.\end{definition}

This is property (CW) in \cite[Lemma~3.4.1]{Fu1}. It has been shown in \cite{Fu2} that (CW) is equivalent to the property~(C) of consistency as defined in~\cite{Fu1}:
\begin{itemize}
\item[\rm(C)] Every strong Cauchy sequence in $\mathcal E_\kappa^+$ converges strongly to any of its vague cluster
points.\end{itemize}

\begin{definition}A (positive definite) kernel $\kappa$ is said to be {\it perfect\/} if it is consistent and strictly positive definite; or equivalently, if
$\mathcal E_\kappa^+$ is strongly complete and the strong topology on $\mathcal E_\kappa^+$ is finer than the induced vague topology on $\mathcal E_\kappa^+$ (see \cite[p.~166 and Theorem~3.3]{Fu1}).\end{definition}

\begin{remark}Even for a perfect kernel $\kappa$ the whole pre-Hilbert space $\mathcal E_\kappa$ is, in general, strongly {\it incomplete\/}, and this is the case also for the Coulomb kernel $|x-y|^{-1}$ on $\mathbb R^3$, $|x-y|$ being the Euclidean distance between $x$ and $y$
(cf.~\cite{Ca}). Compare with \cite[Theorem~1]{ZUmzh} where the strong completeness has been established for the metric subspace of all signed measures $\nu\in\mathcal E_{\kappa_\alpha}(\mathbb R^n)$, $n\geqslant3$, such that $\nu^+$ and $\nu^-$ are supported by closed nonintersecting sets $A_1,A_2\subset\mathbb R^n$; here $\kappa_\alpha(x,y)=|x-y|^{\alpha-n}$ is the $\alpha$-Riesz kernel of order $\alpha\in(0,n)$. This result from \cite{ZUmzh} has been proved with the aid of Deny's  theorem \cite{D1} stating that $\mathcal E_{\kappa_\alpha}(\mathbb R^n)$ can be completed by making use of tempered distributions on $\mathbb R^n$ with finite
$\alpha$-Riesz energy.\end{remark}

The property of consistency (or perfectness) is particularly useful in minimum energy problems over subclasses of~$\mathcal E_\kappa$. E.g., if $Q$ is a closed set with $c_\kappa(Q)\in(0,+\infty)$ and $\kappa$ is a consistent kernel, then the infimum in (\ref{cap}) is an actual minimum. If, moreover, $\kappa$ is perfect, then the corresponding minimizing measure is unique. See \cite[Theorem~4.1]{Fu1}.

\section{$\alpha$-Riesz sweeping in $\mathbb R^n$}\label{sec2}
Throughout this section we fix $n>2$, $n\in\mathbb N$, and $\alpha\in(0,2]$, and consider the $\alpha$-Riesz kernel $\kappa_\alpha$. We shall simply write $\alpha$ instead of $\kappa_\alpha$ if it serves as an index. For example, $c_\alpha(\cdot)$ denotes the $\alpha$-Riesz inner capacity of a set. In all that follows, `{\it n.e.}' ({\it nearly everywhere\/}) means that a proposition involving a variable point holds everywhere except for a subset with $c_\alpha(\cdot)=0$.\footnote{To be precise, one should write `$c_\alpha$-n.e.' instead of `n.e.', but for the sake of brevity we shall always use the latter short form.
This will not cause any misunderstanding, for the order $\alpha$ of the Riesz kernel is fixed.}

We denote by $\omega_{\mathbb R^n}$ the {\it Alexandroff point\/}  of $\mathbb R^n$, and write
$\overline{\mathbb R^n}:=\mathbb R^n\cup\{\omega_{\mathbb R^n}\}$.

When speaking of a positive (Radon) measure $\mu\in\mathfrak M^+=\mathfrak M^+(\mathbb R^n)$ on $\mathbb R^n$, we always assume that $\kappa_\alpha\mu\not\equiv+\infty$.
This implies that
\begin{equation}\label{1.3.10}\int_{|y|>1}\,\frac{d\mu(y)}{|y|^{n-\alpha}}<+\infty,\end{equation}
cf.\ \cite[Eq.~1.3.10]{L}, and consequently, $\kappa_\alpha\mu$ is finite n.e.\ on $\mathbb R^n$, cf.\ \cite[Chap\-ter~III, Section~1]{L}; these two implications can actually be reversed. Under these (permanent) requirements, $\kappa_\alpha$ satisfies all the principles (i)--(iv) from Section~\ref{sec-pr} and it is perfect; see Theorems 1.7, 1.10, 1.15, 1.18, 1.27 and 1.29 in~\cite{L}.

Throughout this section $A$ is a {\it closed\/} proper subset of $\mathbb R^n$. To avoid triviality, we shall always assume that $c_\alpha(A)>0$.

\subsection{$\alpha$-thinness at $y\in\overline{\mathbb R^n}$. $\alpha$-Riesz equilib\-rium measure in an extended sense}\label{sec-thin}
A point $y\in A$ is said to be
$\alpha${\it -irregular\/} if and only if $A$ is $\alpha${\it -thin\/} at $y$, that is, $A$ is thin at $y$ in the sense of Brelot \cite{Br-Pisa}, applied to the cone of all positive $\alpha$-superharmonic functions on $\mathbb R^n$ (together with the constant function $+\infty$). All others points of $A$ are said to be $\alpha${\it -reg\-ular\/}. See also \cite[Chapter~V, Section~3, n$^\circ$\,9]{L}. Regarding the notion of $\alpha${\it -superharmonic\/} function, see \cite[Chapter~I, Section~6, n$^\circ$\,20]{L}.

Alternatively, by Wiener's criterion \cite[Theorem~5.2]{L}, $y\in A$ is $\alpha$-ir\-reg\-ular if and only if
\begin{equation}\label{W}\sum_{k\in\mathbb N}\,\frac{c_\alpha(A_k)}{q^{k(n-\alpha)}}<+\infty,\end{equation}
where $q\in(0,1)$ and $A_k:=A\cap\bigl\{x\in\mathbb R^n:\ q^{k+1}\leqslant|x-y|<q^k\bigr\}$.
Denote by $A_I=A_{I,\alpha}$ the set of all $\alpha$-irregular points of $A$; then $A_I\subset\partial_{\mathbb R^n}A$ and $c_\alpha(A_I)=0$, cf.\ \cite[Lemma~5.2]{L}.

It follows from the perfectness of $\kappa_\alpha$ and Frostman's maximum principle that, for any (closed) set $A\subset\mathbb R^n$ with $c_\alpha(A)<+\infty$, there exists a unique {\it equilibrium measure\/} $\gamma_A=\gamma_{A,\alpha}\in\mathcal E_\alpha^+(A)$ on~$A$, possessing the following properties:
\begin{align}\label{sec2-1}\|\gamma_A\|_\alpha^2&=\gamma_A(A)=c_\alpha(A),\\
\label{sec2-2}\kappa_\alpha\gamma_A&=1\quad\text{n.e.\ on \ }A,\\
\label{sec2-3}\kappa_\alpha\gamma_A&\leqslant1\quad\text{everywhere on \ }\mathbb R^n,\\
\label{sec2-4}\kappa_\alpha\gamma_A&=\min_{\theta\in\Theta_A}\,\kappa_\alpha\theta,
\end{align}
where $\Theta_A$ consists of all $\theta\in\mathfrak M^+(\mathbb R^n)$ with $\kappa_\alpha\theta\geqslant1$ n.e.\ on~$A$;
see \cite[Theorem~4.1]{Fu1}, \cite[Theorem 2.6]{L} and also \cite[Lemma 4.5]{L}.

We extend the notion of $\alpha$-Riesz equilibrium measure to an (unbounded closed) set $A$ with $c_\alpha(A)=+\infty$. Following~\cite{Z2} (or~\cite{Br} for $\alpha=2$), we call $A$ $\alpha${\it -thin\/} at the Alexandroff point $\omega_{\mathbb R^n}$ if the inverse $A^*$ of $A\cup\{\omega_{\mathbb R^n}\}$ relative to the $(n-1)$-dim\-en\-sional unit sphere $S(y,1)$ centered at $y\in\mathbb R^n$ is $\alpha$-thin at $y$ as defined above, or equivalently if either $y\not\in A^*$ or $y$ is an $\alpha$-irregular point of $A^*$, cf.\ \cite[Theorem~5.10]{L}. The notion of $\alpha$-thinness of $A$ at $\omega_{\mathbb R^n}$ does not depend on the choice of~$y$, cf.~\cite{Z2}. If (and only if) $A$ is $\alpha$-thin at $\omega_{\mathbb R^n}$ there exists $\gamma_A\in\mathfrak M^+(A)$ possessing the properties (\ref{sec2-2})--(\ref{sec2-4}) (see \cite[Chapter~V, Section~1, n$^\circ$\,1]{L}). It is clear from the proof of \cite[Theorems~5.1]{L} that such $\gamma_A$ is $c_\alpha$-absolutely continuous and hence it is unique (cf.\ \cite[p.~178, Remark]{L}). Furthermore, (\ref{sec2-1}) also holds in the sense that all its three terms are~$+\infty$.\footnote{Equality (\ref{sec2-2}) in fact holds everywhere on $A\setminus A_I$, cf.\ Corollary~\ref{reg-com}.}

\subsection{$\alpha$-Riesz sweeping: definition and statements on existence and uniqueness}
Throughout this section, fix $\mu\in\mathfrak M^+$. We first consider the case where $\mu$ has finite energy.

\begin{theorem}\label{th-bala-f}
For\/ $\mu\in\mathcal E^+_\alpha$ and\/ $A$ closed in\/ $\mathbb R^n$ there exists\/ $\mu^A\in\mathcal E_\alpha^+(A)$ such that
\begin{align}\label{eq-bala-f1}
\kappa_\alpha\mu^A&=\kappa_\alpha\mu\quad\text{n.e.\ on \ }A,\\
\kappa_\alpha\mu^A&\leqslant\kappa_\alpha\mu\quad\text{everywhere\ on \ }\mathbb R^n.\label{eq-bala-f2}
\end{align}
Such\/ $\mu^A$ is actually the orthogonal projection of\/ $\mu$ in the pre-Hilbert space\/ $\mathcal E_\alpha$ onto the convex cone\/ $\mathcal E^+_\alpha(A)$, i.e.\footnote{See, e.g., \cite[Chapter~III, Sections~8--10]{Ca2} and \cite[Proposition~1.12.4]{E2}.}
\[\|\mu-\mu^A\|_\alpha<\|\mu-\nu\|_\alpha\quad\text{for all \ }\nu\in\mathcal E^+_\alpha(A), \ \nu\ne\mu^A,\]
and\/ $\mu^A$ is determined uniquely within\/ $\mathcal E_\alpha^+(A)$ by relation\/~{\rm(\ref{eq-bala-f1})}.
\end{theorem}

\begin{proof} Indeed, since $\kappa_\alpha$ is perfect, this can be obtained by generalizing arguments from \cite[Chapter~IV, Section~5, n$^\circ$\,22]{L} (cf.\ also \cite[pp.~243--244]{Ca2} for $\alpha=2$;  in \cite{Ca2,L}, $A=K$ was assumed to be compact). Actually, this has been proved more generally in a locally compact space $X$ for any quasiclosed set  and any perfect kernel $\kappa$ satisfying $\kappa$-domination principle, cf.\ \cite[Theorem~4.12]{Fu5}.\end{proof}

\begin{remark}One could equally well write `q.e.' (quasi everywhere) instead of `n.e.' in relation~(\ref{eq-bala-f1}), where `q.e.' refers to the {\it outer\/} $\alpha$-Riesz capacity of a set, \cite[p.~143]{L}. Indeed,  $\psi:=\kappa_\alpha\mu^A-\kappa_\alpha\mu$, being the difference between two l.s.c.\ functions, is Borel measurable, which yields that the set $\{x\in A:\ \psi(x)\ne0\}$ is $\kappa_\alpha${\it-cap\-acit\-able\/} (see~\cite[Theorem~30.1]{Ch0}, \cite[Theorem~4.5]{Fu1}, \cite[Theorem~2.8]{L}). A similar remark applies to relation~(\ref{sec2-2}) as well.\end{remark}

\begin{corollary}\label{rest-fin}Let\/ $F$ be a closed subset of\/ $A$ with\/ $c_\alpha(F)>0$. Then, in the notations of the preceding theorem,
\begin{equation}\label{muFA}\mu^F=(\mu^A)^F\quad\text{for every \ }\mu\in\mathcal E^+_\alpha.\end{equation}
\end{corollary}

\begin{proof}By Theorem \ref{th-bala-f}, both $\mu^F$ and $(\mu^A)^F$ belong to $\mathcal E_\alpha^+(F)$ and moreover
\[\kappa_\alpha(\mu^A)^F=\kappa_\alpha\mu^A=\kappa_\alpha\mu\quad\text{n.e.\ on \ }F.\]
Since relation (\ref{eq-bala-f1}) (for $F$ instead of $A$) determines $\mu^F$ uniquely among $\mathcal E_\alpha^+(F)$, identity (\ref{muFA}) follows.\end{proof}

Likewise as it has been done in the paragraph before \cite[Theorem 4.16$'$]{L}, (\ref{eq-bala-f1}) implies for $A$ closed
\begin{equation}\label{fine}
\kappa_\alpha(\mu^A,\lambda)=\kappa_\alpha(\mu,\lambda^A)\quad\text{for all \ }\mu,\lambda\in\mathcal E^+_\alpha.
 \end{equation}

Next, we follow Cartan \cite[p.\ 257]{Ca2} and use the symmetry relation (\ref{fine}) to define sweeping of an arbitrary $\mu\in\mathfrak M^+$.

\begin{definition}\label{def-swpt}For\/ $\mu\in\mathfrak M^+$ we call\/ $\mu^A\in\mathfrak M^+(A)$ a {\it balay\'ee\/} of\/ $\mu$ onto\/~$A$ if
\begin{equation}\label{alternative}
\kappa_\alpha(\mu^A,\lambda)=\kappa_\alpha(\mu,\lambda^A)\quad\text{for all \ }\lambda\in\mathcal E^+_\alpha,
 \end{equation}
where for every\/ $\lambda\in\mathcal E^+_\alpha$, $\lambda^A\in\mathcal E^+_\alpha(A)$ is determined uniquely by Theorem\/~{\rm\ref{th-bala-f}}.\end{definition}

\begin{remark}\label{ba-finite}In view of this definition and relation~(\ref{fine}) the measure $\mu^A\in\mathcal E_\alpha^+(A)$ from Theorem~\ref{th-bala-f} may now be called the {\it balay\'ee\/} of $\mu\in\mathcal E^+_\alpha$ onto $A$.\end{remark}

\begin{theorem}\label{th-bala-2}For any\/ $\mu\in\mathfrak M^+$ there exists a unique balay\'ee\/ $\mu^A\in\mathfrak M^+(A)$.
Furthermore, such\/ $\mu^A$ satisfies both relations\/ {\rm(\ref{eq-bala-f1})} and\/ {\rm(\ref{eq-bala-f2})}.
\end{theorem}

\begin{proof} This follows from Theorem~\ref{th-bala-f} and relation (\ref{fine}) with arguments similar to those in the proof of \cite[Theorem~4.16$'$]{L} (now for a closed set $A$ instead of a compact $K$). Indeed, likewise as in \cite[p.~272]{L} (see also \cite[p.~257, footnote]{Ca2}) for $\mu\ne0$ one can construct  a sequence of measures $\mu_k\in\mathcal E^+_\alpha$ such that $\kappa_\alpha\mu_k\uparrow\kappa_\alpha\mu$ and $\mu_k\to\mu$ vaguely (as $k\to+\infty$). Then, applying relation~(\ref{eq-bala-f1}) to $\mu_k\in\mathcal E^+_\alpha$, cf.\ Theorem~\ref{th-bala-f}, we obtain
\[\kappa_\alpha\mu_k^A=\kappa_\alpha\mu_k\leqslant\kappa_\alpha\mu_{k+1}=\kappa_\alpha\mu_{k+1}^A\]
n.e.\ on $A$ and hence $\mu_k^A$-a.e., for $\mu_k^A\in\mathcal E_\alpha^+(A)$. By the $\kappa_\alpha$-dom\-in\-ation principle \cite[Theorems~1.27, 1.29]{L},
\[\kappa_\alpha\mu_k^A\leqslant\kappa_\alpha\mu_{k+1}^A\quad\text{everywhere on \ }\mathbb R^n.\]
Thus, $\kappa_\alpha\mu_k^A$ increases along with $\kappa_\alpha\mu_k$ and does not exceed $\kappa_\alpha\mu$. According to \cite[Theorem~3.9]{L}, there exists $\nu\in\mathfrak M^+$ such that
\[\kappa_\alpha\mu_k^A\uparrow\kappa_\alpha\nu\]
and $\mu_k^A\to\nu$ vaguely (as $k\to+\infty$). Since $A$ is closed, the latter implies $\nu\in\mathfrak M^+(A)$. Besides, having written relation (\ref{fine}) for $\mu_k\in\mathcal E^+_\alpha$ and then applied
\cite[Chapter~IV, Section~1, Th\'eor\`eme~3]{B2}, we arrive at (\ref{alternative}) with $\mu^A:=\nu$. The measure $\mu^A\in\mathfrak M^+(A)$ constructed just above is thus a balay\'ee of $\mu\in\mathfrak M^+$ onto~$A$, and $\kappa_\alpha\mu_k^A\uparrow\kappa_\alpha\mu^A$. Now, having written relations (\ref{eq-bala-f1}) and (\ref{eq-bala-f2}) for $\mu_k\in\mathcal E^+_\alpha$, cf.\ Theorem~\ref{th-bala-f}, and then letting $k\to\infty$, we arrive at relations (\ref{eq-bala-f1}) and (\ref{eq-bala-f2}) for $\mu\in\mathfrak M^+$ as claimed.

For uniqueness, having assumed that (\ref{alternative}) also holds for some $\nu'\in\mathfrak M^+(A)$ in place of $\mu^A$, we conclude that, for any $r>0$,
\[\kappa_\alpha\mu^A\ast m^{(r)}=\kappa_\alpha\nu'\ast m^{(r)},\]
where $m^{(r)}$ is the measure obtained by uniformly distributing unit mass over the open ball $B(0,r):=\{x\in\mathbb R^n: \ |x|<r\}$ and $\ast$ denotes the convolution. Letting $r\to0$ in the last display and applying  \cite[Theorems~1.11, 1.12]{L} establishes $\nu'=\mu^A$.\end{proof}

\begin{corollary}\label{ineq-en} For any\/ $\mu\in\mathfrak M^+$ we have\/ $\kappa_\alpha(\mu^A,\mu^A)\leqslant\kappa_\alpha(\mu,\mu)$.
\end{corollary}

\begin{proof}Applying relation (\ref{eq-bala-f2}), cf.\ Theorem~\ref{th-bala-2}, we obtain
\[\kappa_\alpha(\mu^A,\mu^A)\leqslant\kappa_\alpha(\mu,\mu^A)=\kappa_\alpha(\mu^A,\mu)\leqslant\kappa_\alpha(\mu,\mu)\]
as claimed.\end{proof}

Finally, the symmetry relation (\ref{fine}) is extended to arbitrary $\mu,\nu\in\mathfrak M^+$.

\begin{theorem}\label{th-alt2}For any\/ $\mu,\nu\in\mathfrak M^+$ we have
\begin{equation}\label{alternative2}
\kappa_\alpha(\mu^A,\nu)=\kappa_\alpha(\mu,\nu^A).
\end{equation}
\end{theorem}

\begin{proof}It is seen from the proof of Theorem~\ref{th-bala-2} that $\kappa_\alpha\nu^A$ is the pointwise limit of an increasing sequence $\kappa_\alpha\nu_k^A$, $k\in\mathbb N$, where $\nu_k\in\mathcal E_\alpha^+$ and $\kappa_\alpha\nu_k\uparrow\kappa_\alpha\nu$ (as $k\to+\infty$). Hence, by (\ref{alternative}) for $\nu_k$ in place of $\lambda$,
\[\kappa_\alpha(\mu^A,\nu_k)=\kappa_\alpha(\mu,\nu_k^A)\quad\text{for all \ }k\in\mathbb N.\]
Letting $k\to+\infty$ and applying \cite[Chapter~IV, Section~1, Th\'eor\`eme~3]{B2}, we thus get (\ref{alternative2}), as was to be proved.\end{proof}

\begin{corollary}\label{cor-sym}For any\/ $x,y\in\mathbb R^n$,
\begin{equation}\label{symD}\kappa_\alpha\varepsilon_x^A(y)=\kappa_\alpha(\varepsilon_x^A,\varepsilon_y)=\kappa_\alpha(\varepsilon_x,\varepsilon^A_y)=\kappa_\alpha\varepsilon_y^A(x),\end{equation}
where\/ $\varepsilon_z$ denotes the unit Dirac measure at a point\/ $z\in\mathbb R^n$. More generally, for every\/ $\mu\in\mathfrak M^+$ and every\/ $y\in\mathbb R^n$,
\begin{equation}\label{symD2}\kappa_\alpha\mu^A(y)=\kappa_\alpha(\mu^A,\varepsilon_y)=\kappa_\alpha(\mu,\varepsilon^A_y)=\int\kappa_\alpha\varepsilon^A_y(x)\,d\mu(x).\end{equation}
\end{corollary}

\begin{proof}Indeed, both (\ref{symD}) and (\ref{symD2}) follow directly from (\ref{alternative2}).\end{proof}

\begin{lemma}\label{l-char} Given\/ $\mu\in\mathfrak M^+$ and\/ $A$, the swept potential\/ $\kappa_\alpha\mu^A$ {\rm(}and, hence, the swept measure\/ $\mu^A${\rm)} can be characterized uniquely by the relation
\begin{equation}\label{bal-char}\kappa_\alpha\mu^A=\min\,\kappa_\alpha\xi,\end{equation} where\/ $\xi\in\mathfrak M^+$ ranges over all measures with
\begin{equation}\label{eq-char}\kappa_\alpha\xi\geqslant\kappa_\alpha\mu\quad\text{n.e.\ on \ }A.\end{equation}
\end{lemma}

\begin{proof}Since the swept measure $\mu^A$ satisfies relation (\ref{eq-char}) in accordance with (\ref{eq-bala-f1}), cf.\ Theorem~\ref{th-bala-2}, it is enough to show that
\begin{equation}\label{eq-char1}\kappa_\alpha\mu^A\leqslant\kappa_\alpha\xi\quad\text{everywhere on \ }\mathbb R^n\end{equation}
for every $\xi\in\mathfrak M^+$ possessing the property (\ref{eq-char}). As seen from the proof of Theorem~\ref{th-bala-2}, $\kappa_\alpha\mu^A$ is the pointwise limit of an increasing sequence $\kappa_\alpha\mu_k^A$, $k\in\mathbb N$, where $\mu_k\in\mathcal E_\alpha^+$ and $\kappa_\alpha\mu_k\uparrow\kappa_\alpha\mu$ (as $k\to+\infty$). Since
\[\kappa_\alpha\mu_k^A\leqslant\kappa_\alpha\mu_k\leqslant\kappa_\alpha\mu\leqslant\kappa_\alpha\xi\]
n.e.\ on $A$ and hence $\mu_k^A$-a.e., the $\kappa_\alpha$-domination principle \cite[Theorems~1.27, 1.29]{L} yields $\kappa_\alpha\mu_k^A\leqslant\kappa_\alpha\xi$ on all of $\mathbb R^n$. Letting here $k\to+\infty$ leads to relation (\ref{eq-char1}).\end{proof}

\subsection{Properties of the swept measure.~I} Our next goal is to show that sweeping of a positive measure does not increase the total mass. Actually, the following more general statement holds.

\begin{theorem}\label{cor-mass'} {\rm(Principle of positivity of mass\footnote{The principle of positivity of mass has been introduced by Deny \cite[p.~165]{D2}.})} For any\/ $\mu,\nu\in\mathfrak M^+$ such that\/ $\kappa_\alpha\mu\geqslant\kappa_\alpha\nu$ everywhere on\/ $\mathbb R^n$ we have\/ $\mu(\mathbb R^n)\geqslant\nu(\mathbb R^n)$. In particular,
\begin{equation}\label{t-mass}\mu(\mathbb R^n)\geqslant\mu^A(\mathbb R^n)\quad\text{for any \ }\mu\in\mathfrak M^+.\end{equation}
\end{theorem}

\begin{proof}
Consider the sequence of the closed balls $\overline{B}_k:=\overline{B}(0,k):=\{x\in\mathbb R^n: \ |x|\leqslant k\}$, $k\in\mathbb N$, and let $\gamma_k$ be the $\alpha$-Riesz equilibrium measure on $\overline{B}_k$.
Then $1=\kappa_\alpha\gamma_k=\kappa_\alpha\gamma_{k+1}$ everywhere on $\overline{B}_k$, cf.\ \cite[Chapter~II, Section~3, n$^\circ$\,13]{L}, and by the $\kappa_\alpha$-domination principle $\kappa_\alpha\gamma_k\leqslant\kappa_\alpha\gamma_{k+1}$
on all of $\mathbb R^n$. Thus the sequence $\kappa_\alpha\gamma_k$, $k\in\mathbb N$, is increasing, clearly with the pointwise limit~$1$. For $\mu,\nu\in\mathfrak M^+$ with $\kappa_\alpha\mu\geqslant\kappa_\alpha\nu$ everywhere on $\mathbb R^n$, it follows that
\[\int\kappa_\alpha\gamma_k\,d\nu=\int\kappa_\alpha\nu\,d\gamma_k\leqslant
\int\kappa_\alpha\mu\,d\gamma_k=\int\kappa_\alpha\gamma_k\,d\mu,\]
whence the former part of the theorem by letting $k\to+\infty$. Taking here $\mu^A$ instead of $\nu$, which is possible in view of (\ref{eq-bala-f2}), we obtain relation~(\ref{t-mass}).\end{proof}

The latter part of Theorem~\ref{cor-mass'} is specified by Theorem~\ref{bal-mass-th} below.

\begin{theorem}\label{equiv} For any\/ $\alpha$-reg\-ular point\/ $y\in A$ we have\/ $\varepsilon_y^A=\varepsilon_y$. For any other\/ $y\in\mathbb R^n$, $\varepsilon_y^A$ is\/ $c_\alpha$-ab\-sol\-utely continuous.\end{theorem}

\begin{proof}We first need to recall the well-known notion of Kelvin transform of measures (see \cite{R} and \cite[pp.~260--261]{L}).

Define the inversion with respect to $S(y,1)$ mapping each point $x\ne y$ to the point~$x^*$ on the ray through~$x$ issuing from~$y$ which is determined uniquely by
\[|x-y|\cdot|x^*-y|=1.\]
This is a homeomorphism of $\mathbb
R^n\setminus\{y\}$ onto itself; furthermore,
\begin{equation}\label{inv}|x^*-z^*|=\frac{|x-z|}{|x-y||z-y|}.\end{equation}
It can be extended to a homeomorphism of $\overline{\mathbb R^n}$ onto itself such that $y$ and $\omega_{\mathbb R^n}$ are mapped to each other.

To each $\nu\in\mathfrak M$ with
$\nu(\{y\})=0$ we assign the Kelvin transform
$\nu^*\in\mathfrak M$ by means of the formula
\begin{equation}\label{kelv-m}d\nu^*(x^*)=|x-y|^{\alpha-n}\,d\nu(x),\quad x^*\in\mathbb R^n.\end{equation}
Then, in view of (\ref{inv}),
\begin{equation}\label{KP}\kappa_\alpha\nu^*(x^*)=|x-y|^{n-\alpha}\kappa_\alpha\nu(x),\quad x^*\in\mathbb R^n,\end{equation}
and therefore
\begin{equation}\label{K}\kappa_\alpha(\mu^*,\nu^*)=\kappa_\alpha(\mu,\nu)\end{equation}
for every $\mu\in\mathfrak M$ with $\mu(\{y\})=0$. The last display is obtained by multiplying (\ref{kelv-m}) (with $\mu$ in place of $\nu$) by (\ref{KP}) and next integrating with respect to $d\mu(x)$ over~$\mathbb R^n$.\footnote{Each of equalities (\ref{KP}) and (\ref{K}) is understood in the sense that the value on the left is well-defined if (and only if) so is that on the right, and then they coincide.}
Furthermore, by (\ref{kelv-m}), $\nu^*(\mathbb R^n)=\kappa_\alpha\nu(y)$, which in view of the relation $(\nu^*)^*=\nu$ proves the equality
\begin{equation}\label{kelv-mmm}\nu(\mathbb R^n)=\kappa_\alpha\nu^*(y).\end{equation}

For the proof of Theorem~\ref{equiv}, fix a point $y\in\mathbb R^n$ and consider $A^*$, the inverse of $A\cup\{\omega_{\mathbb R^n}\}$ with respect to $S(y,1)$.
Having assumed that $y$ is an $\alpha$-regular point of $A$ we first assert that then $\varepsilon_y^A(\{y\})>0$.  Indeed, if not, then by (\ref{KP}) the Kelvin transform $\bigl(\varepsilon_y^A\bigr)^*$ of $\varepsilon_y^A$ has the $\alpha$-Riesz potential equal to $1$ n.e.\ on $A^*\cap\mathbb R^n$, which means that $\bigl(\varepsilon_y^A\bigr)^*$ is the $\alpha$-Riesz equilibrium measure on $A^*\cap\mathbb R^n$, treated in the sense of \cite[Chapter~V, Section~1]{L}. Hence, $A^*\cap\mathbb R^n$ is $\alpha$-thin at $\omega_{\mathbb R^n}$, cf.\ Section~\ref{sec-thin}, which contradicts the $\alpha$-regularity of~$y$. We next proceed by proving that the relation $\varepsilon_y^A(\{y\})>0$ thus obtained yields $\varepsilon_y^A=\varepsilon_y$. Indeed, if not, then $\varepsilon_y^A=c\varepsilon_y+\chi$, where $\chi\in\mathfrak M^+(A\setminus\{y\})$, $\chi\ne0$, and $0<c<1$, the latter inequality being clear from relation (\ref{t-mass}) applied to $\mu=\varepsilon_y$.  Then, by (\ref{eq-bala-f1}), cf.\ Theorem~\ref{th-bala-2},
\[|x-y|^{\alpha-n}=\kappa_\alpha\varepsilon_y^A(x)=c|x-y|^{\alpha-n}+\kappa_\alpha\chi(x)\quad\text{n.e.\ on \ }A,\]
hence $\kappa_\alpha\chi_1(x)=|x-y|^{\alpha-n}$ n.e.\ on $A$, where $\chi_1:=\chi/(1-c)$. Since $\chi_1(\{y\})=0$, (\ref{KP}) applied to $\nu=\chi_1$ shows that the Kelvin transform of $\chi_1$ is the equilibrium measure on $A^*\cap\mathbb R^n$, which is impossible by the $\alpha$-regularity of $y$.

To establish the latter statement of the theorem, suppose first that $y\in A$ is $\alpha$-irregular. Then the (unbounded closed) set $A^*\cap\mathbb R^n$ is $\alpha$-thin at $\omega_{\mathbb R^n}$ and hence there exists the equilibrium measure $\gamma_{A^*}\in\mathfrak M^+(A^*\cap\mathbb R^n)$ on $A^*\cap\mathbb R^n$, which is characterized uniquely by relations (\ref{sec2-1})--(\ref{sec2-4}) with $A^*$ in place of $A$.
Denoting by $\delta$ the Kelvin transform of $\gamma_{A^*}$, we conclude from (\ref{sec2-2}) (with $\gamma_{A^*}$ instead of $\gamma_A$) and (\ref{KP}) that
\[\kappa_\alpha\delta(x)=|x-y|^{\alpha-n}=\kappa_\alpha\varepsilon_y(x)\quad\text{n.e.\ on \ }A.\]
Here we have used the fact that the assertions $c_\alpha(E^*)=0$ and $c_\alpha(E)=0$, $E\subset A$, are equivalent, cf.\ \cite[p.~261]{L}. This observation also yields that $\delta$ is $c_\alpha$-absolutely continuous along with $\gamma_{A^*}$. Using (\ref{sec2-4}) (with $\gamma_{A^*}$ instead of $\gamma_A$), we also observe that $\delta$ satisfies (\ref{bal-char}) for $\varepsilon_y$ in place of $\mu$, and so the ($c_\alpha$-absolutely continuous) measure $\delta$ is, in fact, the swept measure $\varepsilon_y^A$.

Finally, suppose that $y\in\mathbb R^n\setminus A$. Then the inverse $A^*$ of $A\cup\{\omega_{\mathbb R^n}\}$ is a compact subset of $\mathbb R^n$ containing $y$. The rest of the proof runs in the same way as in the preceding paragraph, even with the standard notion of the $\alpha$-Riesz equilibrium measure~$\gamma_{A^*}$.\end{proof}

\begin{remark}\label{remark3.12} Theorem~\ref{equiv} is a particular case of results obtained in \cite{BH} in the very general setting of balayage spaces. The former assertion follows from \cite[Chapter~VII, Proposition 3.1]{BH} and the latter by combining \cite[Chapter~VI, Proposition 5.6]{BH} and \cite[Chapter~VII, Proposition 4.1]{BH}. We have, however, chosen to bring the above alternative proof based on the Kelvin transform because we want to make a presentation of our results based on a single approach, while for this purpose the general balayage theory is insufficient anyway, cf.\ the Introduction for details. Moreover, the relation between $\varepsilon_y^A$ and the $\alpha$-Riesz equilibrium measure $\gamma_{A^*}$ of $A^*$, given with the aid of the Kelvin transform, is decisive for the proof of Theorem \ref{bal-mass-th} below.
\end{remark}

\begin{corollary}\label{cor-bal-reg}For any\/ $\mu\in\mathfrak M^+$ we have
\begin{equation}\label{eq-reg1}\kappa_\alpha\mu^A=\kappa_\alpha\mu\quad\text{everywhere on \ }A\setminus A_I.\end{equation}
\end{corollary}

\begin{proof}
Indeed, for every $\alpha$-regular point $y\in A$, $\varepsilon_y^A=\varepsilon_y$ by Theorem~\ref{equiv}, and therefore $\kappa_\alpha\mu^A(y)=\kappa_\alpha\mu(y)$ by (\ref{symD2}).\end{proof}

\begin{corollary}\label{reg-com} Assume\/ $A$ to be\/ $\alpha$-thin at\/ $\omega_{\mathbb R^n}$. Then
\begin{equation}\label{eq-reg2}\kappa_\alpha\gamma_{A,\alpha}=1\quad\text{everywhere on \ }A\setminus A_I.\end{equation}
\end{corollary}

\begin{proof}Fix $y\notin A$ and consider the inversion with respect to $S(y,1)$. It follows from \cite[Chapter~IV, Section~5, n$^\circ$\,19]{L} (see the first two displays on p.~261 therein) and Wiener's criterion (\ref{W}) that then $A\setminus A_I$ is mapped onto $A^*\setminus A^*_I$, where $A^*$ is the inverse of $A\cup\{\omega_{\mathbb R^n}\}$. As seen from the proof of Theorem~\ref{equiv} (with $A$ replaced by $A^*$), the equilibrium measure $\gamma_{A,\alpha}$ is the Kelvin transform of the swept measure $\varepsilon_y^{A^*}$. Combined with equalities (\ref{KP}) and (\ref{eq-reg1}) this establishes (\ref{eq-reg2}).\end{proof}

If now $\nu\in\mathfrak M(\mathbb R^n)$ is a {\it signed\/} (Radon) measure, then $\nu^A:=(\nu^+)^A-(\nu^-)^A$ is said to be a {\it balay\'ee\/} of $\nu$ onto the (closed) set $A$. The balay\'ee $\nu^A$ is unique, for so are $(\nu^+)^A$ and $(\nu^-)^A$, and it is supported by $A$.
Its $\alpha$-Riesz potential $\kappa_\alpha\nu^A$ is well-defined and finite n.e.\ on $\mathbb R^n$, and $\kappa_\alpha\nu^A(x)=\kappa_\alpha\nu(x)$ at every $x\in A\setminus A_I$ where either of $\kappa_\alpha\nu^\pm(x)$ is finite, cf.\ Corollary~\ref{cor-bal-reg}.

\subsection{$\mu$-adequate family of measures. Integral representation of $\mu^A$}\label{adequate1} For the notion of a $\mu${\it -adequate\/} family of measures, see \cite[Section~3, D\'efinition~1]{Bou}. Write $D:=A^c$.

\begin{lemma}\label{adequate} For every $\mu\in\mathfrak M^+(D)$ the family $(\varepsilon_y^A)_{y\in D}$ is $\mu$-adeq\-uate, that is,
\begin{itemize}
\item[\rm{(a)}] for any function $f\in C_0(\mathbb R^n)$ the numerical function $y\mapsto\int f\,d\varepsilon_y^A$ on $D$ is essentially $\mu$-in\-tegrable;
\item[\rm{(b)}] the map $y\mapsto\varepsilon_y^A$ is vaguely $\mu$-measurable on $D$.
\end{itemize}
\end{lemma}

\begin{proof} Fix $\mu\in\mathfrak M^+(D)$.

(a) Essential integrability over $D$ is the same as integrability because the locally compact space $D$ is countable at $\omega_D$, the Alexandroff point of~$D$ (cf.\ \cite[Section~2, Proposition~3]{Bou}).

Suppose to begin with that $f\in C_0^{\infty}(\mathbb R^n)$. As in \cite[Lemma~1.1]{L} define a function $\psi=\kappa_{-\alpha}\ast f$, which amounts to $f=\kappa_{\alpha}\psi$. The convolution $\psi$ of the distribution $\kappa_{-\alpha}$ with $f\in C_0^\infty(\mathbb R^n)$ is a $C^\infty$-function, by \cite[Th\'eor\`eme~XI]{S}.
According to \cite[Eq.~1.3.16]{L}, $\psi(x)=O(|x|^{-n-\alpha})$ as $|x|\to+\infty$. It follows that
\begin{equation}\label{psi}
\psi^{\pm}(x)\leqslant C\min\,\bigl\{1,|x|^{-n-\alpha}\bigr\},
\end{equation}
$C$ denoting a constant (not necessarily the same at each occurrence).
Denote by $\nu$ the measure on $\mathbb R^n$ with density $\psi^\pm$: $d\nu(x)=\psi^\pm(x)\,dx$, where $dx$ refers to the $n$-dimensional Lebesgue measure. We begin by proving that
$\nu\in\mathcal E^+_\alpha$. Denote by $\overline{B}=\overline{B}(0,1)$ the closed unit ball in $\mathbb R^n$ and by $\nu_0$ and $\nu_1$ the restrictions of $\nu$ to $\overline{B}$ and $\overline{B}\,^c$, respectively. Then
$\kappa_\alpha\nu_0=\kappa_\alpha\ast(1_{\overline{B}}\psi^{\pm})$, $1_{\overline{B}}$ being the indicator function for $\overline{B}$, is bounded on $\overline{B}$, and hence $\nu_0$ has finite energy $\kappa_\alpha(\nu_0,\nu_0)$. Furthermore, $\kappa_\alpha\nu_0(x)=O\bigl(|x|^{\alpha-n}\bigr)$ as $|x|\to+\infty$,
and so altogether
  \begin{equation}\label{3.24a}\kappa_\alpha\nu_0(x)\leqslant C\min\,\bigl\{1,|x|^{\alpha-n}\bigr\}.
  \end{equation}

Next, let $\nu_1^*$ denote the image of $\nu_1$ under Kelvin transformation with respect to the unit circle $S(0,1)$ (noting that  $\nu_1(\{0\})=0$). By (\ref{kelv-m}) and (\ref{KP}) (both with $y=0$),
 \begin{equation}\label{kelvin}
  d\nu_1^*(x^*)=|x|^{\alpha-n}\,d\nu_1(x),
\quad\kappa_\alpha\nu_1^*(x^*)=|x|^{n-\alpha}\kappa_\alpha\nu_1(x).
\end{equation}
Hence $\nu_1$ and $\nu_1^*$ have the same $\alpha$-Riesz energy, cf.\  (\ref{K}).
According to inequality (\ref{psi}),
\[d\nu_1^*(x^*)=|x|^{\alpha-n}1_{\overline{B}\,^c}(x)\psi^{\pm}(x)\,dx\leqslant C|x|^{\alpha-n}|x|^{-\alpha-n}\,dx=C|x|^{-2n}\,dx=C\,dx^*,
\]
the latter equality being valid because $|x|^{-n}\,dx=|x^*|^n\,dx^*$. In fact, write $x=r\xi$ with $r=|x|$ and where $\xi$ ranges over the unit sphere $S(0,1)$ endowed with its surface measure $d\xi$. We obtain $dx=r^{n-1}\,dr\,d\xi$ and similarly $dx^*=(r^*)^{n-1}\,dr^*\,d\xi$ with $r^*=r^{-1}$, hence $dr^*=-r^{-2}\,dr$. We may neglect the minus sign (change of orientation) and conclude that indeed $dx^*=|x|^{-2n}\,dx$.

Thus the situation for $\nu_1^*$ is essentially the same as above for $\nu_0$, both being supported by the ball $\overline{B}$ and having a bounded density, and so
\[\kappa_\alpha\nu_1^*(x^*)\leqslant C\min\,\bigl\{1,|x^*|^{\alpha-n}\bigr\}\]
and hence, by the latter equation (\ref{kelvin}),
\[\kappa_\alpha\nu_1(x)=|x|^{\alpha-n}\kappa_\alpha\nu_1^*(x^*)
\leqslant C\min\,\bigl\{1,|x|^{\alpha-n}\bigr\}.\]
When combined with inequality (\ref{3.24a}) this leads to
\begin{equation}\label{kappa}
  \kappa_\alpha\nu(x)\leqslant C\min\,\bigl\{1,|x|^{\alpha-n}\bigr\}.
\end{equation}
In particular, $\kappa_\alpha\nu_1(x)\leqslant C|x|^{\alpha-n}$ on $\overline{B}\,^c$ and
\[\kappa_\alpha(\nu_1,\nu_1)=\int\kappa_\alpha\nu_1\,d\nu_1\leqslant C\int|x|^{\alpha-n}\,d\nu_1(x)=C\kappa_\alpha\nu_1(0)<+\infty.\]
As $\nu=\nu_0+\nu_1$, we thus get
\begin{equation}\label{ad1}\nu\in\mathcal E_\alpha^+.\end{equation}

Identifying the measures $\psi^+\,dx$ and $\psi^-\,dx$ with their densities $\psi^+$ and $\psi^-$, respectively,  we obtain
\[\int f\,d\varepsilon_y^A=\int\kappa_\alpha\psi^+\,d\varepsilon_y^A-\int\kappa_\alpha\psi^-\,d\varepsilon_y^A=\kappa_\alpha(\psi^+)^A(y)-\kappa_\alpha(\psi^-)^A(y)
\]
according to (\ref{alternative}) applied to $\mu=\varepsilon_y$ and $\lambda=\nu=\psi^\pm\,dx$. The last member in the above display is the difference between two finite l.s.c.\ functions. For the proof that $y\mapsto\int f\,d\varepsilon_y^A=\kappa_\alpha\psi^A$ is $\mu$-integrable it suffices to show that $\int\kappa_\alpha\nu\,d\mu<+\infty$. According to inequality (\ref{kappa}) we obtain
\[\int_{\overline{B}}\kappa_\alpha\nu\,d\mu\leqslant C\int_{\overline{B}}\,d\mu<+\infty\]
and
\[\int_{{\overline{B}}\,^c}\kappa_\alpha\nu\,d\mu
\leqslant C\int_{{\overline{B}}\,^c}|x|^{\alpha-n}\,d\mu(x)=C\int_{\overline{B}}\,d\mu^*(x^*)<+\infty,\]
the equality being valid by the former equation (\ref{kelvin}) with $\nu_1$ replaced by $1_{{\overline{B}}^c}\mu$, assuming that $\mu\bigl(\{0\}\bigr)=0$. If $\mu\bigl(\{0\}\bigr)>0$ we remove the mass at $0$ from $\mu$, which does not affect the $\mu$-integrability of the finite valued function $\kappa_\alpha\nu$.

For general $f\in C_0(\mathbb R^n)$, or just as well $f\in C_0^+(\mathbb R^n)$, we regularize $f$
in the standard way, as in \cite[p.~22]{S}, thereby obtaining a sequence of positive functions $f_j\in C_0^\infty(\mathbb R^n)$ supported by a fixed compact neighborhood of the support of $f$ and such that $f_j$ converges uniformly to $f$. Since for every $y\in\mathbb R^n$, $\varepsilon_y^A(\mathbb R^n)\leqslant1$ by inequality (\ref{t-mass}), it follows that the sequence $\int f_j\,d\varepsilon_y^A$ converges uniformly on $\mathbb R^n$ to $\int f\,d\varepsilon_y^A$. As shown above, each of the functions $y\mapsto\int f_j\,d\varepsilon_y^A$ is $\mu$-integrable, and so is therefore their uniform limit $\int f\,d\varepsilon_y^A$ (see \cite[Chapter~IV, Section~3, Proposition~4]{B2}).

(b) For the proof that the map
$D\ni y\mapsto\varepsilon_y^A\in\mathfrak M^+(D)$ is vaguely $\mu$-meas\-urable, cf.\ \cite[Section~3, n$^\circ$\,1]{Bou}, it suffices according to \cite[Section~1, n$^\circ$\,2]{Bou} to show that this map is {\it vaguely continuous\/} on $D$. (As pointed out in \cite[p.~18, Remarque]{Bou} it is not enough to verify that each of the functions $y\mapsto\varepsilon_y^A$ is $\mu$-measurable, as it is done in  \cite[p.~214, footnote~12]{L}.)

Likewise as in the proof of assertion (a) above, consider first a function $f\in C_0^\infty(\mathbb R^n)$ and choose a signed measure $\psi\in\mathcal E_\alpha$ so that $\kappa_\alpha\psi=f$. According to relation (\ref{ad1}), $\psi^{\pm}\in\mathcal E^+_\alpha$, which in view of (\ref{alternative}) for $\mu=\varepsilon_y$ and $\lambda=\psi^\pm$ yields
\[\int f\,d\varepsilon_y^A
=\int\kappa_\alpha\psi\,d\varepsilon_y^A
=\int\kappa_\alpha\psi^A\,d\varepsilon_y
=\kappa_\alpha\psi^A(y).
\]
When varying $y$, $\kappa_\alpha\psi^A(y)$ is a (finite and) continuous function of $y\in D$ (because $(\psi^{\pm})^A$ is supported by~$A$), and so is therefore $\int f\,d\varepsilon_y^A$ in the present case $f\in C_0^\infty(\mathbb R^n)$. But the same holds for
any $f\in C_0(\mathbb R^n)$. Indeed, likewise as above, one may choose a sequence of $C_0^\infty$-fun\-ctions $f_j$ on $\mathbb R^n$ converging uniformly to the given function $f\in C_0(\mathbb R^n)$. Then, by relation (\ref{t-mass}) for $\mu=\varepsilon_y$,
 \[\Bigl|\int(f-f_j)\,d\varepsilon_y^A\Bigr|\leqslant\sup_j\,|f-f_j|\to0\quad\text{(as $j\to+\infty$)},\]
and so $\int f\,d\varepsilon_y^A$ is indeed a (finite) continuous function of $y\in D$, being the uniform limit of the continuous functions $\int f_j\,d\varepsilon_y^A$ on~$D$.\end{proof}

\begin{theorem}\label{th-int-rep}For any $\mu\in\mathfrak M^+(D)$, we have the integral representation
  \begin{equation}\label{eq-int-rep}\mu^A=\int\varepsilon_y^A\,d\mu(y).
\end{equation}\end{theorem}

\begin{proof}
Fix $\mu\in\mathfrak M^+(D)$. Since, by Lemma~\ref{adequate}, the family of measures $(\varepsilon_y^A)_{y\in D}$ is $\mu$-adequate we may according to \cite[Section~3, n$^\circ$\,2]{Bou} define the integral $\nu=\int\varepsilon_y^A\,d\mu(y)$ by
$$
\int f(z)\,d\nu(z)=\int\Bigl(\int f(z)\,d\varepsilon_y^A(z)\Bigr)\,d\mu(y),
$$
$f\in C_0(\mathbb R^n)$ being arbitrary. According to \cite[Section~3, Proposition~1]{Bou} this identity remains valid when $f$ is allowed to be any positive l.s.c.\ function on $\mathbb R^n$ (the integrals being then understood as upper integrals).\footnote{For still more general integrands  see \cite[Section 4, Th\'eor\`eme 1]{Bou}.} For given $x\in\mathbb R^n$ we apply this to $f(z)=\kappa_\alpha(x,z)$, $z\in\mathbb R^n$:
\begin{equation}\label{repr-th1}
\kappa_\alpha\nu(x)=\int\Bigl(\int\kappa_\alpha(x,z)\,d\varepsilon_y^A(z)\Bigr)\,d\mu(y)=\int\kappa_\alpha\varepsilon_y^A(x)\,d\mu(y).
\end{equation}
To establish (\ref{eq-int-rep}) it remains to show that $\nu=\mu^A$, that is,
\[\kappa_\alpha(\nu,\lambda)=\kappa_\alpha(\mu,\lambda^A)\quad\text{for every \ }\lambda\in\mathcal E^+_\alpha,\]
cf.\ Definition~\ref{def-swpt}. Applying (\ref{alternative}) with $\varepsilon_y$ in place of $\mu$ and (\ref{repr-th1}) we get by Fubini's theorem
\begin{align*}\kappa_\alpha(\nu,\lambda)&=\int\kappa_\alpha\nu(x)\,d\lambda(x)=\int\Bigl(\int\kappa_\alpha\varepsilon_y^A(x)\,d\mu(y)\Bigr)\,d\lambda(x)\\
  {}&={\int\Bigl(\int\kappa_\alpha\varepsilon_y^A(x)\,d\lambda(x)\Bigr)\,d\mu(y)
  =\int\Bigl(\int\kappa_\alpha\varepsilon_y(x)\,d\lambda^A(x)\Bigr)\,d\mu(y)}\\
  {}&=\int\Bigl(\int\kappa_\alpha(x,y)\,d\mu(y)\Bigr)\,d\lambda^A(x)
  =\int\kappa_\alpha\mu\,d\lambda^A
  =\kappa_\alpha(\mu,\lambda^A),
\end{align*}
as claimed.\end{proof}

\begin{remark}\label{La} An assertion similar to Theorem \ref{th-int-rep} can be found in \cite[Chapter~V, Section~1]{L}, but the proof given there is incomplete, as noted above in the proof of Lemma~\ref{adequate}.
\end{remark}

\subsection{Properties of the swept measure.~II}\label{sec-II} Based on the results obtained above, we proceed with analyzing properties of the $\kappa_\alpha$-swept measure $\mu^A$. Recall that $D$ denotes the complement of $A$ to $\mathbb R^n$.

\begin{corollary}\label{C}For any\/ $\mu\in\mathfrak M^+(D)$, $\mu^A$ is\/ $c_\alpha$-absolutely continuous.\end{corollary}

\begin{proof}Consider a compact set $K\subset\mathbb R^n$ with $c_\alpha(K)=0$; then for any $y\in D$, $\varepsilon_y^A(K)=0$ by the latter assertion of Theorem~\ref{equiv}. Applying \cite[Section~3, Th\'eor\`eme~1]{Bou}, we then conclude from (\ref{eq-int-rep}) that
\[\int 1_K\,d\mu^A=\int\,d\mu(y)\int 1_K(x)\,d\varepsilon_y^A(x)=0,\]
and so $\mu^A$ is indeed $c_\alpha$-absolutely continuous.\end{proof}

\begin{corollary}\label{C1}For any\/ $\mu\in\mathfrak M^+(D)$, $\mu^A$ is determined uniquely by relation\/ {\rm(\ref{eq-bala-f1})} among the\/ $c_\alpha$-absolutely continuous positive measures supported by\/~$A$.\end{corollary}

\begin{proof}The balay\'{e}e $\mu^A$ is $c_\alpha$-absolutely continuous, by Corollary~\ref{C}, and satisfies relation (\ref{eq-bala-f1})  according to Theorem~\ref{th-bala-2}. If $\nu\in\mathfrak M^+(A)$ possesses these two properties then $\kappa_\alpha\nu=\kappa_\alpha\mu^A$ n.e.\ on $A$, and an application of \cite[p.~178, Remark]{L} results in $\nu=\mu^A$.\end{proof}

\begin{corollary}\label{rest'}For every\/ $\mu\in\mathfrak M^+(D)$ and every closed subset\/ $F$ of\/ $A$ with\/ $c_\alpha(F)>0$,
\[\mu^F=(\mu^A)^F.\]
\end{corollary}

\begin{proof}Indeed, the assertion follows from Corollaries~\ref{C} and~\ref{C1} in a way similar to that in the proof of Corollary~\ref{rest-fin}.\end{proof}

The following assertion specifies the latter part of Theorem \ref{cor-mass'}.

\begin{theorem}\label{bal-mass-th} For\/ $A$ to be\/ $\alpha$-thin at\/ $\omega_{\mathbb R^n}$ it is necessary and sufficient that there exists a non-ze\-ro bounded measure\/ $\mu\in\mathfrak M^+(D)$ such that
\begin{equation}\label{t-mass'}\mu^A(\mathbb R^n)<\mu(\mathbb R^n).\end{equation}
If moreover $D$ is connected then the inequality\/ {\rm(\ref{t-mass'})} holds for every non-ze\-ro bounded\/ $\mu\in\mathfrak M^+(D)$ \rm{(}provided that\/ $A$ is\/ $\alpha$-thin at\/ $\omega_{\mathbb R^n}${\rm)}.\footnote{For $\alpha<2$ the latter assertion of the theorem remains valid even if the requirement of connectedness of $D$ is omitted.}
\end{theorem}

\begin{proof} To prove the sufficiency part of the former assertion of the theorem, assume that on the contrary $A$ is not $\alpha$-thin at $\omega_{\mathbb R^n}$. Fix $y\in D$, and let $A^*$ be the inverse of
$A\cup\{\omega_{\mathbb R^n}\}$ with respect to $S(y,1)$. Then $A^*$ is a compact set, and $y\in A^*$ is an $\alpha$-regular point of $A^*$. According to Corollary~\ref{reg-com}, we have $\kappa_\alpha\gamma_{A^*}(y)=1$, $\gamma_{A^*}$ being the equilibrium measure on $A^*$. For the Kelvin transform $\nu$ of $\gamma_{A^*}$, we thus conclude from relations (\ref{KP})--(\ref{kelv-mmm}) that $\nu\in\mathcal E^+_\alpha(A)$ because $\gamma_{A^*}\in\mathcal E^+_\alpha(A^*)$, and that
\begin{equation}\label{1}\nu(\mathbb R^n)=
\kappa_\alpha\gamma_{A^*}(y)=1,\end{equation}
and also that
\begin{equation}\label{11}\kappa_\alpha\nu(x)=|x-y|^{\alpha-n}\kappa_\alpha\gamma_{A^*}(x^*)=\kappa_\alpha\varepsilon_y(x)\quad\text{for nearly all \ }x\in A,\end{equation}
the last display being valid in view of the fact that the assertions $c_\alpha(E^*)=0$ and $c_\alpha(E)=0$, $E\subset A$, are equivalent, cf.\ \cite[Chapter~IV, Section~5, n$^\circ$\,19]{L}.
Since $\varepsilon_y^A$ is $c_\alpha$-absolutely continuous according to the latter assertion of Theorem~\ref{equiv}, relation (\ref{11}) yields $\nu=\varepsilon_y^A$, cf.\  Corollary~\ref{C1}. Hence, by (\ref{1}), $\varepsilon_y^A(\mathbb R^n)=1$. Combined with (\ref{eq-int-rep}) this gives for every $\mu\in\mathfrak M^+(D)$
\[\mu^A(\mathbb R^n)=\int d\mu^A=\int\,d\mu(y)\int d\varepsilon_y^A(x)=\mu(\mathbb R^n),\]
cf.\ \cite[Section~3, Th\'eor\`eme~1]{Bou}, and the sufficiency part of the theorem follows.

If now $A$ is $\alpha$-thin at $\omega_{\mathbb R^n}$, then there exists the unique (in general unbounded) $c_\alpha$-ab\-sol\-utely continuous $\alpha$-Riesz equilibrium measure $\gamma_A\in\mathfrak M^+(A)$. One can choose a connected component $D_i$ of $D$ so that
$\kappa_\alpha\gamma_A\not\equiv1$ on $D_i$, for if not then $\kappa_\alpha\gamma_A$ equals $1$ everywhere on $D$, hence n.e.\ on $\mathbb R^n$, cf.\ (\ref{sec2-2}).
Thus $\gamma_A$ serves also as the $\alpha$-Riesz equilibrium measure on $\mathbb R^n$, so that $\mathbb R^n$ itself is $\alpha$-thin at $\omega_{\mathbb R^n}$. Contradiction.

We proceed by showing that, for the given $D_i$,
\begin{equation}\label{strless}
\kappa_\alpha\gamma_A<1\quad\text{everywhere on \ }D_i.
\end{equation}
 On the contrary, let this not hold; then by inequality (\ref{sec2-3}) $\kappa_\alpha\gamma_A(x_0)=1$ at some $x_0\in D_i$. Fix an open neighborhood $U$ of $x_0$ so that $C\ell_{\mathbb R^n}U\subset D_i$. Then both $\kappa_\alpha\gamma_A$ and $1$ are $\alpha$-super\-har\-monic on $\mathbb R^n$, $\alpha$-harmonic on $U$, and continuous on $C\ell_{\mathbb R^n}U$, cf.\  \cite[Theorem~1.4]{L} for $\alpha=2$ and \cite[Chapter~I, Section~6, n$^\circ$\,20]{L} for $\alpha<2$.
Since, in consequence of relation (\ref{sec2-3}), $\kappa_\alpha\gamma_A$ takes its maximum value at $x_0$, we infer from \cite[Theorems~1.1, 1.28]{L} that $\kappa_\alpha\gamma_A=1$ everywhere on $U$, hence everywhere on $D_i$, which contradicts the choice of $D_i$.

The theorem will be established once we have shown that inequality (\ref{t-mass'}) holds for every non-ze\-ro bounded $\mu\in\mathfrak M^+(D_i)$.
Since both $\gamma_A$ and $\mu^A$ are $c_\alpha$-ab\-sol\-utely continuous, cf.\ Corollary~\ref{C}, we thus have, by relations (\ref{sec2-2}), (\ref{eq-bala-f1}), cf.\ Theorem~\ref{th-bala-2}, and (\ref{strless}),
\[\mu^A(\mathbb R^n)=\kappa_\alpha(\mu^A,\gamma_A)=\kappa_\alpha(\mu,\gamma_A)<\mu(D_i)\leqslant\mu(\mathbb R^n),\]
as was to be proved.\end{proof}

\begin{remark}Theorem \ref{bal-mass-th} has been announced in earlier papers of the second named author (see \cite[Theorem~4]{Z2}; for $\alpha=2$, see also \cite[Theorem~B]{Z0}). Since in both these papers the integral representation from \cite{L} was essentially used, we consider it pertinent to provide here an independent proof, cf.\ Remark~\ref{La}.\end{remark}

\section{$\alpha$-Green kernel}\label{sec4}

In all that follows, consider a fixed domain $D\subset\mathbb R^n$ with the complement $A:=D^c:=\mathbb R^n\setminus D$, and the (generalized) $\alpha${\it -Green kernel\/} $g=g_D^\alpha$ on $D$ defined by
\[g^\alpha_D(x,y)=\kappa_\alpha\varepsilon_y(x)-\kappa_\alpha\varepsilon_y^A(x)\quad
\text{for all \ }x,y\in D.\]
The second term on the right is called the {\it compensating term\/} for $g$.

The properties of the $\alpha$-Green kernel $g=g_D^\alpha$, to be given below, generalize those of the $\alpha$-Riesz kernel, corresponding to the case $c_\alpha(A)=0$.

\subsection{Basic properties of the $\alpha$-Green kernel}

It is seen from (\ref{symD}) that the compensating term is {\it symmetric\/}, and so is therefore $g$, that is,\ $g(x,y)=g(y,x)$ for all $x,y\in D$.
Furthermore, $\kappa_\alpha\varepsilon_y^A(x)$ is (finite and) continuous as a function of $(x,y)\in D\times D$ (see \cite{Z}). It follows that $g$ is l.s.c.\ on $D\times D$, continuous off the diagonal, and takes the value~$+\infty$ on the diagonal. Thus, the $\alpha$-Green kernel $g=g_D^\alpha$ is a (positive function) kernel on the locally compact space $X=D$ (see Section~\ref{sec-ker}).

For any $Q\subset D$, the assertions $c_\alpha(Q)=0$ and $c_g(Q)=0$, $c_g(\cdot)$ being the inner capacity relative to the kernel~$g$, are equivalent, cf.\ \cite[Lemma~2.6]{Z}. Therefore, if some statement $\mathcal U(x)$ is valid n.e.\ on $B\subset D$, then $c_g(N)=0$, $N$ consisting of all $x\in B$ with $\mathcal U(x)$ not to hold; and also the other way around.

\begin{lemma}\label{lem-posit'}$g^\alpha_D(x,y)>0$ for every\/ $(x,y)\in D\times D$.\end{lemma}

\begin{proof} Suppose that, on the contrary, $g(x,y)=0$ at some $(x,y)\in D\times D$. Then
\[\kappa_\alpha\varepsilon_{y}(x)=\kappa_\alpha\varepsilon_{y}^A(x).\]
Consider an open neighborhood $U\subset D$ of $y$ such that $C\ell_{\mathbb R^n}U\subset D$.
Then $\kappa_\alpha\varepsilon_{y}^A(\cdot)$ is continuous on $D$ (and, hence, on $C\ell_{\mathbb R^n}U$) and $\alpha$-harmonic on $D$, while $\kappa_\alpha\varepsilon_{y}$ is $\alpha$-superharmonic on $\mathbb R^n$, cf.\  \cite[Theorem~1.4]{L} for $\alpha=2$ and \cite[Chapter~I, Section~6, n$^\circ$\,20]{L} for $\alpha<2$. According to relations (\ref{eq-bala-f1}) and (\ref{eq-bala-f2}), cf.\ Theorem~\ref{th-bala-2},
\[\kappa_\alpha\varepsilon_{y}^A\leqslant\kappa_\alpha\varepsilon_{y}\quad\text{everywhere on \ }\mathbb R^n,\]
the equality being valid n.e.\ on~$A$.
In view of the last two displays we therefore conclude from \cite[Theorems~1.1, 1.28]{L} that $\kappa_\alpha\varepsilon_y^A=\kappa_\alpha\varepsilon_{y}$ a.e.\ on $\mathbb R^n$.
By \cite[Theorem 1.12]{L}, this yields $\varepsilon_y^A=\varepsilon_y$, which is impossible.\end{proof}

\begin{remark}Lemma \ref{lem-posit'} can actually be strengthened by \cite[Theorem~3.4]{Ku}, noting that for an open ball $B$ such that $C\ell_{\mathbb R^n}B\subset D$, $g^\alpha_D(x,y)\geqslant g^\alpha_B(x,y)$, $x,y\in B$, the latter being clear from Corollary~\ref{rest'}.	
\end{remark}

\begin{definition}\label{d-ext}A measure $\nu\in\mathfrak M(D)$ is called {\it extendible\/} if its extension by~$0$ to $\mathbb R^n$, denoted again by $\nu$, is a (Radon) measure on $\mathbb R^n$ such that (\ref{1.3.10}) holds for both $\nu^+$ and $\nu^-$.\end{definition}

We identify an extendible measure $\nu\in\mathfrak M(D)$ with its extension by~$0$ to $\mathbb R^n$. A measure $\nu\in\mathfrak M(D)$ is extendible if and only if $|\nu|$ is extendible (or equivalently $\nu^+$ and $\nu^-$ are so).
Every bounded measure is of course extendible. The converse holds if $D$ is bounded, but not in general (e.g., not if $A$ is compact).

\begin{lemma}\label{l-hatg}For any extendible measure\/ $\nu\in\mathfrak M(D)$, $g\nu$ is well-defined and finite n.e.\ on $D$ and given by
\begin{equation}\label{hatg}g\nu=\kappa_\alpha\nu-\kappa_\alpha\nu^A.\end{equation}
\end{lemma}

\begin{proof}It is seen from Definition~\ref{d-ext} that $\kappa_\alpha\nu$ is finite n.e.\ on $\mathbb R^n$, cf.\ the beginning of Section~\ref{sec2}, and hence so is $\kappa_\alpha\nu^A$. By identity (\ref{repr-th1}), applied to $\nu^\pm$, we get
 \[g\nu(x)=\int\,\bigl[\kappa_\alpha\varepsilon_y(x)-\kappa_\alpha\varepsilon_y^A(x)\bigr]\,d\nu(y)=\kappa_\alpha\nu(x)-\kappa_\alpha\nu^A(x)\]
for nearly every $x\in D$, and the lemma follows.\end{proof}

\begin{lemma}\label{l-hen}If\/ $\nu\in\mathcal E_g(D)$ is extendible, then
\begin{equation}\label{eq1-l-hen}\|\nu\|^2_g=\kappa_\alpha(\nu-\nu^A,\nu-\nu^A).\end{equation}
If, moreover, $\nu$ has compact support in\/ $D$, then
$\nu\in\mathcal E_\alpha(\mathbb R^n)$ and
\begin{equation}\label{eq1-2-hen}\|\nu\|^2_g=\|\nu-\nu^A\|^2_\alpha=\|\nu\|^2_\alpha-\|\nu^A\|^2_\alpha.\end{equation}
\end{lemma}

\begin{proof} For the former assertion we observe that, by Lemma~\ref{l-hatg}, $g\nu$ is finite $c_g$-n.e.\ on $D$ and given by (\ref{hatg}). Besides, since $\nu\in\mathcal E_g(D)$,  the same holds $|\nu|$-a.e.\ on $D$, cf.\ \cite[Lemma~2.3.1]{Fu1}. Integrating (\ref{hatg}) with respect to $\nu^\pm$, we therefore obtain by subtraction
\begin{equation}\label{l1}+\infty>g(\nu,\nu)=\kappa_\alpha(\nu-\nu^A,\nu).\end{equation}
As $\kappa_\alpha(\nu-\nu^A)=0$ n.e.\ on $A$ by (\ref{eq-bala-f1}), while $\nu^A$ is $c_\alpha$-absolutely continuous by Corollary~\ref{C}, we also have
\begin{equation}\label{l22}\kappa_\alpha(\nu-\nu^A,\nu^A)=0,\end{equation}
which results in (\ref{eq1-l-hen}) when combined with (\ref{l1}). Furthermore, since $|\nu|$ along with $\nu$ is extendible and has finite $\alpha$-Green energy, we likewise get
relation (\ref{l1}) with $|\nu|$ in place of~$\nu$.

If, moreover, $\nu$ has compact support in $D$, then $\kappa_\alpha(|\nu|,|\nu|^A)$ is finite, because $\kappa_\alpha|\nu|^A$ is continuous on $D$ and hence bounded on the compact set $S_D^\nu$. In view of (\ref{l1}) with $|\nu|$ in place of $\nu$, we thus see that $\nu$ and $\nu^A$ have finite $\alpha$-Riesz energy and, hence, relation (\ref{eq1-l-hen}) is in fact the former equality in (\ref{eq1-2-hen}). Furthermore, then $\|\nu^A\|^2_\alpha=\kappa_\alpha(\nu,\nu^A)$, cf.\ (\ref{l22}), and
the former equality in (\ref{eq1-2-hen}) yields the latter.\end{proof}

\subsection{Potential-theoretic principles for the $\alpha$-Green kernel}

We proceed to show that $g_D^\alpha$ satisfies the domination principle, even in a stronger form which includes the complete maximum principle and hence the Frostman maximum principle.

\begin{theorem}\label{th-dom-pr} Let\/ $\mu\in\mathcal E^+_g$, let\/ $\nu\in\mathfrak M^+(D)$ be an extendible measure, and let\/ $w$ be a positive $\alpha$-superharmonic function on\/ $\mathbb R^n$. Suppose that
\[g\mu\leqslant g\nu+w\quad\mu\text{-a.e.\ on \ }D.\] Then the same inequality holds on all of\/~$D$.\end{theorem}

\begin{proof} Suppose first that $S^\mu_D$ is compact (in $D$) and that $A=D^c$ is compact.
Then both $\mu$ and $\nu$ extend uniquely by $0$ to similarly denoted (Radon) measures on $\mathbb R^n$, and $\kappa_\alpha\mu\not\equiv+\infty$ and $\kappa_\alpha\nu\not\equiv+\infty$ according to Definition~\ref{d-ext}. Applying Lemma~\ref{l-hen} to $\mu$, we get $\mu\in\mathcal E^+_\alpha(\mathbb R^n)$ and $\mu^A\in\mathcal E^+_\alpha(A)$.
Furthermore, (\ref{hatg}) applied to $\mu$ and $\nu$ gives
\[
\kappa_\alpha\mu=\kappa_\alpha\mu^A+g\mu,\quad\kappa_\alpha\nu=\kappa_\alpha\nu^A+g\nu,
\]
and consequently
\begin{align}
\label{eqn-1}\kappa_\alpha(\mu+\nu^A)&=\kappa_\alpha(\mu^A+\nu^A)+g\mu,\\
\kappa_\alpha(\mu^A+\nu)&=\kappa_\alpha(\mu^A+\nu^A)+g\nu\label{eqn-2}
\end{align}
n.e.\ on $D$, and hence from $g\mu\leqslant g\nu+w$ $\mu$-a.e.\ on $D$
\begin{equation}\label{4.9}
\kappa_\alpha(\mu+\nu^A)\leqslant\kappa_\alpha(\mu^A+\nu)+w
\end{equation}
$\mu$-a.e.\ on $D$, and actually $\mu$-a.e.\ on $\mathbb R^n$ because $\mu(A)=0$.

As seen from the proof of Theorem \ref{th-bala-2}, $\kappa_\alpha\nu^A$ is the pointwise limit of an increasing sequence $\kappa_\alpha\nu^A_k$, $k\in\mathbb N$, where $\nu_k\in\mathcal E_\alpha^+(\mathbb R^n)$ and $\kappa_\alpha\nu_k\uparrow\kappa_\alpha\nu$ (as $k\to+\infty$). From relation (\ref{4.9}) we have in particular
\[\kappa_\alpha(\mu+\nu^A_k)\leqslant\kappa_\alpha(\mu^A+\nu)+w\quad\mu\text{-a.e.\ on \ }\mathbb R^n.\]
The same inequality holds $\nu^A_k$-a.e.\ on $D$ since $S^{\nu^A_k}_{\mathbb R^n}\subset A$ and also n.e.\ on $A$ (because so do both relations $\kappa_\alpha\mu=\kappa_\alpha\mu^A$ and $\kappa_\alpha\nu^A_k=\kappa_\alpha\nu_k\leqslant\kappa_\alpha\nu$), and consequently $\nu^A_k$-a.e.\ on $\mathbb R^n$ because $\nu^A_k\in\mathcal E^+_\alpha(A)$. Altogether
\[
\kappa_\alpha(\mu+\nu^A_k)\leqslant\kappa_\alpha(\mu^A+\nu)+w\quad(\mu+\nu^A_k)\text{-a.e.\ on \ }\mathbb R^n.
\]
Since the right-hand member of this inequality is a positive $\alpha$-superharmonic function on $\mathbb R^n$ while $\mu+\nu^A_k\in\mathcal E^+_\alpha(\mathbb R^n)$, we infer by the $\kappa_\alpha$-domination principle \cite[Theorems~1.27, 1.29]{L} followed by making $k\to+\infty$ that
\[
\kappa_\alpha(\mu+\nu^A)\leqslant\kappa_\alpha(\mu^A+\nu)+w\quad\text{everywhere on \ }\mathbb R^n.
\]
Combining this with (\ref{eqn-1}) and (\ref{eqn-2}) and noting that $\kappa_\alpha(\mu^A+\nu^A)<+\infty$ on $D$ leads to
\begin{equation}\label{4.9a}
g\mu\leqslant g\nu+w\quad\text{everywhere on \ }D.\end{equation}

If we drop the above extra hypothesis that $A$ be compact, we choose $y\in D$ neither charging $\nu$ nor $\mu$, and apply the Kel\-vin transformation with respect to $S(y,1)$. Then $A^*$, the inverse of
$A\cup\{\omega_{\mathbb R^n}\}$ with respect to $S(y,1)$, becomes compact; we denote by $D^*$ the (connected) complement of $A^*$ to $\mathbb R^n$. Observe that
\begin{equation}\label{star}(\nu^*)^{A^*}=(\nu^A)^*,\quad (\mu^*)^{A^*}=(\mu^A)^*.\end{equation}
Indeed, by relations (\ref{KP}) and (\ref{eq-bala-f1}),
\begin{align*}\kappa_\alpha(\nu^A)^*(x^*)&=|x-y|^{n-\alpha}\kappa_\alpha\nu^A(x)=|x-y|^{n-\alpha}\kappa_\alpha\nu(x)\\
{}&=\kappa_\alpha\nu^*(x^*)=\kappa_\alpha(\nu^*)^{A^*}(x^*)\end{align*}
for nearly every $x\in A$, or equivalently for nearly every $x^*\in A^*$. Here we have used the fact that the properties $c_\alpha(E^*)=0$ and $c_\alpha(E)=0$, $E\subset A$, are equivalent, cf.\ \cite[Chapter~IV, Section~5, n$^\circ$\,19]{L}. When combined with Corollary~\ref{C} this fact also yields that $(\nu^A)^*$ and $(\nu^*)^{A^*}$ are both $c_\alpha$-absolutely continuous. Therefore, by Corollary~\ref{C1}, the very last display
establishes the former equality (\ref{star}). The proof of the latter is similar.

By Lemma~\ref{l-hen} and identities (\ref{K}) and (\ref{star}), in our assumptions
\begin{align*}+\infty>g^\alpha_D(\mu,\mu)&=\|\mu\|^2_\alpha-\|\mu^A\|^2_\alpha=\|\mu^*\|^2_\alpha-\|(\mu^A)^*\|^2_\alpha\\
{}&=\|\mu^*\|^2_\alpha-\|(\mu^*)^{A^*}\|^2_\alpha=g^\alpha_{D^*}(\mu^*,\mu^*),\end{align*}
so that $g^\alpha_{D^*}(\mu^*,\mu^*)<+\infty$. Furthermore, $\nu^*\in\mathfrak M^+(D^*)$ remains extendible from $D^*$ along with $\nu$ from $D$ since by (\ref{KP}), $\kappa_\alpha\nu^*\not\equiv+\infty$ along with
$\kappa_\alpha\nu$. Besides, by (\ref{hatg}), (\ref{KP}) and (\ref{star}),
\begin{align*}g^\alpha_{D^*}\mu^*(x^*)&=\kappa_\alpha\mu^*(x^*)-\kappa_\alpha(\mu^*)^{A^*}(x^*)\\
  {}&=|x-y|^{n-\alpha}\bigl(\kappa_\alpha\mu(x)-\kappa_\alpha\mu^A(x)\bigr)=|x-y|^{n-\alpha}g_D^\alpha\mu(x)
\end{align*}
and likewise $g^\alpha_{D^*}\nu^*(x^*)=|x-y|^{n-\alpha}g^\alpha_{D}\nu(x)$.

Following Riesz \cite{R} (see also \cite{Ca2} for $\alpha=2$), we define the Kelvin transformation $u^*$ of an $\alpha$-superharmonic function $u$ on $\mathbb R^n$ with respect to $S(y,1)$ by $u^*(x^*)=|x-y|^{n-\alpha}u(x)$; then $(u^*)^*=u$ and $u^*$ is $\alpha$-superharmonic on $\mathbb R^n$, like $u$, cf.\ \cite[pp.~13--14]{R} and \cite[p.~275]{Ca2}. In view of the assumption $g^\alpha_{D}\mu\leqslant g^\alpha_{D}\nu+w$ $\mu$-a.e.\ on $D$, we therefore conclude from the above paragraph that $g^\alpha_{D^*}\mu^*\leqslant g^\alpha_{D^*}\nu^*+w^*$ $\mu^*$-a.e.\ on $D^*$, cf.\ (\ref{kelv-m}). According to what we have already proved, this implies $g^\alpha_{D^*}\mu^*\leqslant g^\alpha_{D^*}\nu^*+w^*$, or equivalently $g\mu\leqslant g\nu+w$ everywhere on $D$, and thus inequality (\ref{4.9a}) as claimed.

Finally, if we also drop the extra hypothesis that $\mu$ have compact support, there is an increasing sequence of compact sets $K$ with the union $D$. For each $K$ we have $g(1_K\mu)\leqslant g\nu+w$ $\mu$-a.e.\ on $D$, in particular $\bigl(1_K\mu\bigr)$-a.e., and therefore everywhere on $D$ as shown above. By varying $K$, the theorem follows.\end{proof}

\begin{remark}\label{rem-pr} The complete maximum principle corresponds to the case where the function $w$ in Theorem~\ref{th-dom-pr} reduces to a constant $c\geqslant0$, the domination principle to that where $w=0$, and the Frostman maximum principle to $\nu=0$ and $w=c$.\end{remark}

\begin{corollary}{\rm(Continuity principle)}\label{cont-pr} If the support\/ $S^\mu_D$ of\/ $\mu\in\mathfrak M^+(D)$ is compact and the restriction of\/ $g\mu$ to\/ $S^\mu_D$ is continuous, then\/ $g\mu$ is continuous on all of\/~$D$.\end{corollary}

\begin{proof} As observed at the beginning of Section 2.2,  $g$ satisfies the continuity principle in consequence of Frostman's maximum principle, cf.\ Theorem~\ref{th-dom-pr} and Remark~\ref{rem-pr}.\end{proof}

\begin{theorem} {\rm(Energy principle)}\label{th-pos-def} $g=g_D^\alpha$ is strictly positive definite.\end{theorem}

\begin{proof} It is enough to consider the case $\alpha\ne2$, for the $2$-Green kernel is strictly positive definite by \cite[Chapter~XIII, Section~7]{Doob}.

As noted at the beginning of Section 2.2 with reference to \cite{N,Ch3}, $g^\alpha_D$ is positive definite in view of the Frostman maximum principle. For strict positive definiteness we refer to the latter part of the proof of \cite[Theorem~2.2]{Z}.\end{proof}

\subsection{Consistency of $g=g^\alpha_D$} We refer to Section~\ref{sec-pr}, Definition~\ref{def-cons}, for the notion of a consistent kernel introduced in \cite{Fu1,Fu2}.

\begin{lemma}\label{l-cons}$g$ is consistent if and only if, for every\/ $\lambda\in\mathcal E^+_g$ of compact support\/ $S^\lambda_D$, the map\/ $\mu\to g(\lambda,\mu)$ is vaguely continuous on the\/ {\rm(}vaguely compact\/{\rm)} truncated cone\/ $\mathcal E^\circ_g$.\end{lemma}

\begin{proof} According to Definition~\ref{def-cons} it suffices to establish the sufficiency part of the assertion.
Fix $\lambda\in\mathcal E_g^+$, and first observe that, for any increasing sequence of compact subsets $K\subset D$ with the union~$D$,
\[g(\lambda,\lambda)
  \leqslant\liminf_{K\uparrow D}\,g(\lambda_K,\lambda_K)
  \leqslant\limsup_{K\uparrow D}\,g(\lambda_K,\lambda_K)
  \leqslant g(\lambda,\lambda),\]
where $\lambda_K$ is the trace of $\lambda$ on~$K$. Indeed, since $g$ is positive and lower semicontinuous on $D\times D$ while $\lambda_K\to\lambda$ vaguely as $K\uparrow D$,  this follows from Lemma~\ref{lemma-semi} (cf.\ also \cite[Lemma~2.2.1]{Fu1}). Hence $\|\lambda\|_g=\lim_{K\uparrow D}\,\|\lambda_K\|_g$, and similarly
$\|\lambda\|^2_g=\lim_{K\uparrow D}\,g(\lambda,\lambda_K)$. Combining these two relations yields
\[\lambda_K\to\lambda\quad\text{strongly in \ } \mathcal E_g.\]

Let now $\mu_i\to\mu$ vaguely as $i\to+\infty$, where $\mu_i,\mu\in\mathcal E^\circ_g$. According to the last display, for any $\varepsilon>0$ there exists a compact set
$K\subset D$ such that $\|\lambda-\lambda_K\|_g<\varepsilon$. For this $K$ choose $i_0$ so that $|g(\lambda_K,\mu_i-\mu)|<\varepsilon$ for all $i\geqslant i_0$.
By the Cauchy--Schwarz (Bunyakovski) inequality, we thus have
\[|g(\lambda,\mu_i)-g(\lambda,\mu)|\leqslant|g(\lambda_K,\mu_i-\mu)|+|g(\mu_i,\lambda_K-\lambda)|+|g(\mu,\lambda_K-\lambda)|<3\varepsilon\]
for all $i\geqslant i_0$, and the lemma follows.\end{proof}

\begin{theorem}\label{th-cons} $g=g_D^\alpha$ is consistent, and hence altogether perfect.\end{theorem}

For $\alpha=2$, perfectness is due to Cartan \cite{Ca} for $D=\mathbb R^n$, $n>2$,
and to Edwards \cite{E} when $D$ is a hyperbolic Riemann surface, in particular a
regular domain in $\mathbb R^n$, $n=2$.
The following proof is inspired by \cite{Ca} and \cite{E}.

\begin{proof} According to \cite[Lemma~3.4.2]{Fu1} it suffices to show that every $\lambda\in\mathcal E_g$ can be approximated strongly by (signed Radon) measures $\mu\in\mathcal E_g$ with $g\mu\in C_0(D)$. Without loss of generality we assume that $\lambda\geqslant0$, and by the proof of Lemma~\ref{l-cons} that $\lambda$ has compact support.

We begin by proving that this measure $\lambda$ can be approximated strongly in $\mathcal E^+_g$ by measures $\lambda_k\in\mathcal E^+_g$, $k\in\mathbb N$, majorized by $\lambda$ and such that the potentials $g\lambda_k$ are bounded and continuous (on $D$). According to the latter part of Lemma~\ref{l-hen} we have $\lambda\in\mathcal E^+_\alpha(\mathbb R^n)$ and hence there exists by \cite[Theorem~3.7]{L} an increasing sequence of measures $\lambda_k\in\mathcal E^+_\alpha(\mathbb R^n)$ possessing the following two properties:
\begin{itemize}
\item[\rm(a)] $\lambda_k\to\lambda$ vaguely and strongly in $\mathcal E_\alpha^+(\mathbb R^n)$,
\item[\rm(b)] $\kappa_\alpha\lambda_k$ belong to $C(\mathbb R^n)$ and $\kappa_\alpha\lambda_k\uparrow\kappa_\alpha\lambda$.
\end{itemize}
It follows that, for any $f\in C^+_0(\mathbb R^n)$,
\[\lambda(f)=\lim_{k}\,\lambda_k(f)\geqslant\lambda_k(f)\quad\text{for every\ }k,\]
and so, indeed, $\lambda_k\leqslant\lambda$. This implies that $\lambda_k$ has compact support $S_{\mathbb R^n}^{\lambda_k}\subset S_D^\lambda$, hence $\kappa_\alpha\lambda_k$ is (continuous and) bounded on $S_{\mathbb R^n}^{\lambda_k}$. Since $\kappa_\alpha\lambda^A_k$ is continuous and bounded on $S_{\mathbb R^n}^{\lambda_k}$ as well, so is $g\lambda_k$. Application of Frostman's maximum principle and the continuity principle for the kernel $g$, cf.\ Theorem~\ref{th-dom-pr}, Remark~\ref{rem-pr}, and Corollary~\ref{cont-pr}, shows that each of $g\lambda_k$, $k\in\mathbb N$, is continuous and bounded on all of~$D$.

Furthermore, as seen from the proof of Theorem~\ref{th-bala-2}, $\kappa_\alpha\lambda_k^A\uparrow\kappa_\alpha\lambda^A$. Since $\lambda^A$ and $\lambda_k^A$, $k\in\mathbb N$, belong to $\mathcal E^+_\alpha(\mathbb R^n)$, it follows from an analogue of \cite[Proposition~4]{Ca} for the (perfect) kernel $\kappa_\alpha$ that $\|\lambda_k^A-\lambda^A\|_\alpha\to0$ (as $k\to+\infty$). Thus, by the latter part of Lemma~\ref{l-hen},
\[\|\lambda_k-\lambda\|_g^2=\|\lambda_k-\lambda\|_\alpha^2-\|\lambda^A_k-\lambda^A\|_\alpha^2\to0\]
as was to be proved.

We may therefore assume from the beginning that, for the given measure $\lambda\in\mathcal E^+_g$ with compact support $S^\lambda_D$ in $D$, both $g\lambda$ and $\kappa_\alpha\lambda$ are bounded and continuous (on $D$ and $\mathbb R^n$, respectively).

We next exhaust $D$ by an increasing sequence of compact sets $L_j$ contained in the
interior $L_{j+1}^\circ$ of $L_{j+1}$ and such that the (closed) sets $F_j:=\mathbb R^n\setminus L_j^\circ$ have no $\alpha$-irregular points.\footnote{For example, let $L_j$ be the (finite) union of all translates of the cube $K_j:=[0,2^{-j}]^n$ by vectors whose coordinates are $2^{-j}$ multiplied by integers $h$ with $|h|\leqslant j$ and such that the translated cubes are contained in $D$. Then $F_j$ consists of only $\alpha$-reg\-ular points, for so does any cube $Q:=[0,a]^n$ in $\mathbb R^n$. In fact, fix $x\in Q$ and choose $0<r<a$ small enough that, for some $i=0,1,2,\dots,n$, $B(x,r)\cap Q$ is one of $2^i$ congruent non-overlapping sets exhausting the ball $B(x,r)$, and hence $c_\alpha\bigl(Q\cap B(x,r)\bigr)\geqslant2^{-i}c_\alpha\bigl(B(x,r)\bigr)$ by subadditivity of $c_\alpha$. By the Wiener criterion in the form used in \cite[p.~289, Eq.~5.1.7]{L} (taking $r:=q^k$, $k\in\mathbb N$, where $0<q<1$), it follows that indeed $x$ is an $\alpha$-reg\-ular point of~$Q$.}
Denote by $\mu_j\in\mathcal E^+_\alpha(F_j)$ the
sweeping of $\lambda\in\mathcal E_\alpha^+(\mathbb R^n)$ onto $F_j$.
Then $\kappa_\alpha\mu_j=\kappa_\alpha\lambda$ everywhere on $F_j$, cf.\ Corollary~\ref{cor-bal-reg}, and consequently $\kappa_\alpha\mu_j$ restricted to $S^{\mu^j}_{\mathbb R^n}$ is continuous. According to the continuity principle for the kernel $\kappa_\alpha$, cf.\ \cite[Theorem~1.7]{L}, $\kappa_\alpha\mu_j$ therefore belongs to $C(\mathbb R^n)$. Furthermore, since $\mu_j$ is the sweeping of $\lambda$ on $F_j$, they both have the same sweeping $\lambda^A=\mu^A_j$ on $A\subset F_j$, cf.\ Corollary~\ref{rest-fin}.

Write $\lambda_j:=\lambda-\mu_{j,D}$ where $\mu_{j,D}$ denotes the trace of $\mu_j$ on $D$. Being bounded, $\lambda$ and $\mu_{j,D}$ are both extendible, and we obtain from Lemma~\ref{l-hatg} on all of $D$,
\begin{align*}g\lambda&=\kappa_\alpha\lambda-\kappa_\alpha\lambda^A=\kappa_\alpha\lambda-\kappa_\alpha\mu_j^A,\\
g\mu_{j,D}&=\kappa_\alpha\mu_{j,D}-\kappa_\alpha\mu_{j,D}^A,\end{align*}
and therefore
\[g\lambda_j=[\kappa_\alpha\lambda-\kappa_\alpha\mu_{j,D}]-[\kappa_\alpha\mu_j^A-\kappa_\alpha\mu_{j,D}^A].\]
But
\[\kappa_\alpha\mu_j^A-\kappa_\alpha\mu_{j,D}^A=\kappa_\alpha\mu_{j,A}^A=\kappa_\alpha\mu_{j,A}\]
everywhere on $\mathbb R^n$ because $\mu_{j,A}^A=\mu_{j,A}$ in consequence of $\mu_{j,A}\in\mathcal E^+_\alpha(A)$. Combining the last two displays gives
\[g\lambda_j=\kappa_\alpha\lambda-\kappa_\alpha\mu_{j,D}-\kappa_\alpha\mu_{j,A}=\kappa_\alpha\lambda-\kappa_\alpha\mu_j\]
on all of $D$, and hence $g\lambda_j$ is indeed of the class $C_0(D)$ since it equals $0$ off the compact set $L_j\subset D$.

It thus remains to show that $\lambda_j\to\lambda$ strongly in $\mathcal E_g$, or equivalently, $\mu_{j,D}\to0$ strongly in $\mathcal E_g$. (Note that $\mu_{j,D}\in\mathcal E^+_g(D)$ since $g\leqslant\kappa_\alpha$ on $D\times D$.)
The proof at this
point in \cite{E} uses $2$-harmonic functions, but cannot be adapted to the present
case $\alpha\leqslant2$ because $\alpha$-har\-mon\-icity for $\alpha<2$ is not a local property.
Instead we use the fact that sweeping of any
measure $\nu\in\mathcal E^+_\alpha(\mathbb R^n)$ onto a closed set $F\subset\mathbb R^n$ amounts to orthogonal projection in the pre-Hil\-bert space $\mathcal E_\alpha(\mathbb R^n)$ onto the convex cone $\mathcal E^+_\alpha(F)$ of all $\mu\in\mathcal E^+_\alpha(\mathbb R^n)$ supported by $F$, cf.\ Theorems~\ref{th-bala-f}, \ref{th-bala-2} and Remark~\ref{ba-finite}. This cone is also strongly closed in $\mathcal E^+_\alpha(\mathbb R^n)$ because $\kappa_\alpha$ is perfect and hence
the strong topology on $\mathcal E^+_\alpha(\mathbb R^n)$ is
finer than the vague topology (cf.\ Definition~\ref{def-cons}).

By Lemma~\ref{l-hen} and the above equality $\mu_{j,A}^A=\mu_{j,A}$,
\begin{equation}\label{end}\|\mu_{j,D}\|^2_g=\|\mu_{j,D}-\mu_{j,D}^A\|^2_\alpha=\|\mu_j-\mu_j^A\|^2_\alpha=\|\mu_j\|^2_\alpha-\|\mu_j^A\|^2_\alpha.\end{equation}
The potentials $\kappa_\alpha\mu_j$, $j\in\mathbb N$,
form a decreasing sequence because, by Corollary~\ref{rest-fin},
\[\mu_{j+1}=\lambda^{F_{j+1}}=(\lambda^{F_j})^{F_{j+1}}=\mu_j^{F_{j+1}}\]
and hence
\[\kappa_\alpha\mu_{j+1}(x)\leqslant\kappa_\alpha\mu_{j}(x)\quad\text{for all \ }x\in\mathbb R^n, \ j\in\mathbb N.\]
As in \cite[Proposition~4]{Ca}, $(\mu_j)$ is therefore Cauchy in $\mathcal E^+_\alpha(\mathbb R^n)$ and hence converges strongly to any of its vague cluster points $\mu$. Since $\mu$ belongs to $\mathcal E^+_\alpha(F_j)$ for every $j$, it is supported by $A=\bigcap_{j}\,F_j$, while
\[\kappa_\alpha\mu=\lim_{j\to+\infty}\,\kappa_\alpha\mu_{j}=\kappa_\alpha\lambda\quad\text{n.e.\ on \ } A,\]
the first equality being valid even n.e.\ on $\mathbb R^n$ (see, e.g., \cite[p.~166, Remark]{Fu1}).
This yields $\lambda^A=\mu$, cf.\ Corollary~\ref{C1}.

Furthermore, $\kappa_\alpha\mu_{j+1}^A=\kappa_\alpha\mu_{j+1}\leqslant\kappa_\alpha\mu_j=\kappa_\alpha\mu_j^A$ n.e.\ on $A$, which according to the $\kappa_\alpha$-dom\-ination principle \cite[Theorems~1.27, 1.29]{L} gives
$\kappa_\alpha\mu_{j+1}^A\leqslant\kappa_\alpha\mu_j^A$ everywhere on $\mathbb R^n$. We thus have the decreasing sequence $\bigl(\kappa_\alpha\mu_j^A\bigr)$ and, likewise as above, an analogue of \cite[Proposition~4]{Ca} for $\kappa_\alpha$ shows that the sequence $\bigl(\mu_j^A\bigr)$ is Cauchy in $\mathcal E^+_\alpha(A)$. Hence, $\mu_j^A\to\lambda^A$ in $\mathcal E^+_\alpha(A)$.
Letting $j\to+\infty$ in (\ref{end}) we see that $\mu_{j,D}\to0$ strongly in $\mathcal E^+_g(D)$ as desired.\end{proof}

\subsection{$\alpha$-Green equilibrium measure. Principle of positivity of mass}

\begin{theorem}\label{th-equi} For any relatively closed subset\/ $F$ of\/ $D$ with\/ $c_g(F)<+\infty$ there exists a unique\/ $\alpha$-Green equilibrium measure on\/ $F$, that is, a measure\/ $\gamma_{F,g}\in\mathcal E^+_g(F)$ such that\/ $\gamma_{F,g}(D)=\|\gamma_{F,g}\|_g^2=c_g(F)$ and
\begin{align}
\label{geq-1}g\gamma_{F,g}&=1\quad\text{n.e.\ on \ }F,\\
g\gamma_{F,g}&\leqslant1\quad\text{everywhere on \ }D.\notag
\end{align}
The measure\/ $\gamma_{F,g}$ is characterized uniquely within\/ $\mathcal E^+_g(F)$ by {\rm(\ref{geq-1})}, and it is the\/ {\rm(}unique\/{\rm)} solution to the problem of minimizing\/ $\alpha$-Green energy over the class\/ $\Gamma_F$ of all\/ $\nu\in\mathcal E_g(D)$ with\/ $g\nu\geqslant1$ n.e.\ on\/ $F$, i.e.
\begin{equation}\label{alt}c_g(F)=\|\gamma_{F,g}\|_g^2=\min_{\nu\in\Gamma_F}\,\|\nu\|_g^2.\end{equation}
Furthermore, relation {\rm(\ref{geq-1})} can be specified as follows:
\begin{equation}\label{geq-2}g\gamma_{F,g}(x)=1\quad\text{for every $\alpha$-regular \ }x\in F.\end{equation}
\end{theorem}

\begin{proof}Except for the very last assertion the stated theorem is obtained from the perfectness of the kernel $g=g^\alpha_D$ (Theorem~\ref{th-cons}) and the Frostman maximum principle (cf.\ Theorem~\ref{th-dom-pr} and Remark~\ref{rem-pr}) in view of \cite[Chapter~II, Section~4.1]{Fu1}.

For the proof of (\ref{geq-2}) one can certainly assume that $c_g(F)>0$, or equivalently  $c_\alpha(F)>0$, cf.\ \cite[Lemma~2.6]{Z}, for if not then there is no $\alpha$-reg\-ular point of $F$. There is also no loss of generality in assuming $c_\alpha(D^c)>0$ since otherwise the relation in question reduces to~(\ref{eq-reg2}).

Assume first that $F=K$ is compact; then $\gamma_{K,g}\in\mathcal E^+_\alpha(K)$ by the latter part of Lemma~\ref{l-hen}.
Consider the $\alpha$-Riesz equilibrium measure $\gamma_{K,\alpha}$ on $K$, and write
\[\chi:=\gamma_{K,\alpha}+(\gamma_{K,g}^A)^K.\]
According to Lemma~\ref{l-hatg} we get from relation (\ref{geq-1}) $\kappa_\alpha\gamma_{K,g}=1+\kappa_\alpha\gamma_{K,g}^A$ n.e.\ on $K$. When combined with (\ref{sec2-2}) and (\ref{eq-bala-f1}) this yields
\[\kappa_\alpha\chi=\kappa_\alpha\gamma_{K,\alpha}+\kappa_\alpha(\gamma_{K,g}^A)^K=\kappa_\alpha\gamma_{K,g}\quad\text{n.e.\ on \ }K.\]
Having observed that $\chi$ and $\gamma_{K,g}$ are both of the class $\mathcal E^+_\alpha(K)$ we thus have $\chi=\gamma_{K,g}$ by
\cite[p.~178, Remark]{L}, and consequently $\kappa_\alpha\chi=\kappa_\alpha\gamma_{K,g}$ everywhere on $\mathbb R^n$, in particular on $K\setminus K_{I,\alpha}$. Applying Corollaries~\ref{cor-bal-reg} and~\ref{reg-com} to $\kappa_\alpha(\gamma_{K,g}^A)^K$ and $\kappa_\alpha\gamma_{\alpha,K}$, respectively, we obtain from the last display
\[\kappa_\alpha\gamma_{K,g}=1+\kappa_\alpha\gamma_{K,g}^A\quad\text{everywhere on \ }K\setminus K_{I,\alpha},\]
which yields (\ref{geq-2}) for $F=K$ compact.

To establish (\ref{geq-2}) for $F$ relatively closed in $D$, consider an increasing sequence of compact sets $K_i\subset F$ such that $\bigcup_i K_i=F$. Then by (\ref{geq-1}), \[1=g\gamma_{K_i,g}=g\gamma_{K_{i+1},g}=g\gamma_{F,g}\quad\text{n.e.\ on \ }K_i,\]
which according to the $g$-domination principle (cf.\ Theorem~\ref{th-dom-pr} and Remark~\ref{rem-pr}) yields
\[g\gamma_{K_i,g}\leqslant g\gamma_{K_{i+1},g}\leqslant g\gamma_{F,g}\quad\text{everywhere on \ }D,\]
and consequently
\begin{equation}\label{greg1}g\gamma_{F,g}(x)\geqslant1\quad\text{for every $\alpha$-regular \ }x\in F.\end{equation}
On the other hand, $\gamma_{K_{i+1},g}\in\Gamma_{K_i}$, and application of \cite[Lemma~4.1.1]{Fu1} gives
\[\|\gamma_{K_{i+1},g}-\gamma_{K_i,g}\|_g^2\leqslant\|\gamma_{K_{i+1},g}\|_g^2-\|\gamma_{K_i,g}\|_g^2.\]
The sequence $\gamma_{K_i,g}\subset\mathcal E^+_g(F)$, $i\in\mathbb N$, is thus Cauchy, and it converges in $\mathcal E_g$ strongly and vaguely to $\gamma_{F,g}$ (see \cite[Proof of Theorem~4.1]{Fu1}). Therefore, by (\ref{geq-2}) applied to $K_i$ and Lemma~\ref{lemma-semi},
\[g\gamma_{F,g}(x)\leqslant\lim_{i\to+\infty}\,g\gamma_{K_i,g}(x)=1\quad\text{for every $\alpha$-regular \ }x\in F.\]
When combined with inequality (\ref{greg1}) this leads to (\ref{geq-2}).\end{proof}

\begin{theorem} \label{cor-mass''} {\rm(Principle of positivity of mass)} For\/ $\mu,\nu\in\mathfrak M^+(D)$ such that\/ $g^\alpha_D\mu\geqslant g^\alpha_D\nu$ everywhere on\/ $D$ we have\/ $\mu(D)\geqslant\nu(D)$.
\end{theorem}

\begin{proof} Likewise as in the proof of Theorem~\ref{th-cons} (see the footnote therein), one can choose an increasing sequence of compact sets $K_i$ without $\alpha$-ir\-reg\-ular points such that $\bigcup_i K_i=D$. Having denoted by $\gamma_i$ the $g$-equilibrium measure on $K_i$ we have $1=g\gamma_i=g\gamma_{i+1}$ everywhere on $K_i$, cf.\ (\ref{geq-2}), and by the $g$-domination principle (cf.\ Theorem~\ref{th-dom-pr} and Remark~\ref{rem-pr}), $g\gamma_i\leqslant g\gamma_{i+1}$
on all of $D$. The rest of the proof runs in a way similar to that in the proof of Theorem~\ref{cor-mass'}.\end{proof}

{\sl Acknowledgements.} The authors thank to Krzysztof Bogdan for useful comments to the paper. The second named author expresses her sincere gratitude to the Department of Mathematical Sciences of the University of Copenhagen for providing a conducive research atmosphere during her stay when part of this manuscript was written.


\begin{thebibliography}{99}

\setlength{\parskip}{1.2ex plus 0.5ex minus 0.2ex}


\bibitem{BH} BLIEDTNER, J., and W.~HANSEN: Potential theory: An analytic and probabilistic approach to balayage. - Springer, Berlin, 1986.

\bibitem{BBC} BOBOC, N., C.~BUCUR, and A.~CORNEA: Order and convexity in potential theory: $H$-cones. - Lectures Notes in Math. 853, Springer,
Berlin, 1981.

\bibitem{B2} BOURBAKI, N.: Int\'egration, chapters~1--4. - Actualit\'es Sci. Ind.
1175, Paris, 1952.

\bibitem{Bou} BOURBAKI, N.: Int\'egration, int\'egration des mesures, chapter~5. - Hermann, Paris, 1956.

\bibitem{Br-Pisa} BRELOT, M.: Introduction axiomatique de l'effilement. - Annali Mat. Pura Appl. 57, 1962, 77--96.

\bibitem{Br} BRELOT, M.: On topologies and boundaries in potential theory. - Lectures Notes in Math. 175, Springer,
Berlin, 1971.

\bibitem{Chen} CHEN, Z.-Q., and R.~SONG: Estimates on Green functions and Poisson kernels for symmetric stable processes. - Math. Ann. 312:3, 1998, 465--501.

\bibitem{C0} CARTAN, H.: Sur les fondements de la th\'eorie du potentiel. - Bull. Soc. Math. France 69, 1941, 71--96.

\bibitem{Ca} CARTAN, H.: Th\'eorie du potentiel newtonien: \'energie, capacit\'e, suites de potentiels. - Bull. Soc. Math. France 73, 1945, 74--106.

\bibitem{Ca2} CARTAN, H.:  Th\'eorie g\'en\'erale du balayage en potentiel newtonien. - Ann. Univ. Grenoble 22, 1946, 221--280.

\bibitem{CD} CARTAN, H., and J.~DENY: Le principe du maximum en th\'eorie du potentiel at la notion de fonction surharmonique. - Acta Sci. Szeged 12, 1950, 81--100.

\bibitem{Ch0} CHOQUET, G.: Theory of capacities. - Ann. Inst. Fourier (Grenoble) 5, 1954, 131--295.

\bibitem{Ch1} CHOQUET, G.: Les noyaux r\'eguliers en th\'eorie du potentiel. - C.R. Acad. Sci. Paris 243, 1956, 635--638.

\bibitem{Ch3} CHOQUET, G.: L'int\'egrale d'\'energie en th\'eorie du potentiel. - S\'em. Th\'eorie Potent., Paris, no.~3, 1958/59.

\bibitem{D1} DENY, J.: Les potentiels d'\'energie finite. - Acta Math. 82, 1950, 107--183.

\bibitem{D2} DENY, J.:  M\'ethodes hilbertiennes en th\'eorie du potentiel. - In: Potential Theory, Centro Internazionale Matematico Estivo, Stresa, 1969, Edizioni Cremonese, Roma, 1970.

\bibitem{Doob} DOOB, J.L.: Classical potential theory and its probabilistic counterpart. - Springer, Berlin, 1984.

\bibitem{Z} DRAGNEV, P.D., B.~FUGLEDE, D.P.~HARDIN, E.B.~SAFF, and N.~ZORII: Minimum Riesz energy problems for a condenser with touching plates. - Potential Anal. 44,  2016, 543--577.

\bibitem{E} EDWARDS, R.E.: Cartan's balayage theory for hyperbolic Riemann surfaces. - Ann. Inst. Fourier 8, 1958, 263--272.

\bibitem{E2} EDWARDS, R.E.: Functional analysis. Theory and applications. - Holt. Rinehart and Winston, New York, 1965.

\bibitem{Fu1} FUGLEDE, B.: On the theory of potentials in locally compact spaces. - Acta Math. 103, 1960, 139--215.

\bibitem{Fu2} FUGLEDE, B.: Caract\'erisation des noyaux consistant en th\'eorie du potentiel. - C.R. Acad. Sci. Paris 255,  1962,
241--243.

\bibitem{Fu4}  FUGLEDE, B.: Capacity as a sublinear functional generalizing an integral. - Mat. Fys. Medd. Dan. Vid. Selsk. 38:7, 1971.

\bibitem{Fu5} FUGLEDE, B.:  Symmetric function kernels and sweeping of measures. - Analysis Math. 42, 2016, 225--259.

\bibitem{Ku} KULCZYCKI, T.: Properties of Green function of symmetric stable processes. - Probab. Math. Statist. 17:2 (1997), 339–-364.


\bibitem{L} LANDKOF, N.S.: Foundations of Modern Potential Theory. - Springer, Berlin, 1972.

\bibitem{N} NINOMIYA, N.: Etude sur la th\'{e}orie du potentiel pris par rapport \`{a} un noyau sym\'{e}trique. - J. Inst. Polytech., Osaka City Univ. Ser.~A. Math.
8, 1957, 147--179.

\bibitem{O1} OHTSUKA, M.: Sur un th\'eor\`eme de M. Kishi. - Proc. Japan Acad. 32, 1956, 722--725.

\bibitem{O2} OHTSUKA, M.: Les relations entre certains principes en th\'eorie du potentiel. - Proc. Japan Acad. 33, 1957, 37--40.

\bibitem{O} OHTSUKA, M.: On potentials in locally compact spaces. - J. Sci. Hiroshima Univ. Ser. A-1 25, 1961, 135--352.

\bibitem{R} RIESZ, M.: Int\'egrales de Riemann--Liouville et potentiels. - Acta Sci. Math. Univ. Szeged 9, 1938, 1--42.

\bibitem{S} SCHWARTZ, L.: Th\'eorie des distributions, vol.~2. - Hermann, Paris, 1951.

\bibitem{Z0} ZORII, N.: An extremal problem of the minimum of energy for space
condensers. - Ukr. Math. J. 38, 1986, 365--369.

\bibitem{Z2} ZORII, N.: A problem of minimum energy for space
condensers and Riesz kernels. - Ukr. Math. J. 41, 1989, 29--36.

\bibitem{ZUmzh} ZORII, N.: A noncompact variational problem in
Riesz potential theory.~I; II. - Ukr. Math. J. 47, 1995, 1541--1553; 48, 1996, 671--682.

\end{thebibliography}
\end{document}